\documentclass[leqno]{amsart}
\usepackage{amssymb}
\usepackage{epic,eepic}
\usepackage[all]{xy}
\usepackage{pstricks}

\newtheorem{theorem}{Theorem}[section]
\newtheorem{lemma}[theorem]{Lemma}
\newtheorem{corollary}[theorem]{Corollary}

\newtheorem{proposition}[theorem]{Proposition}
\theoremstyle{definition}

\newtheorem{example} [theorem] {Example}
\newtheorem{remark} [theorem] {Remark}

\numberwithin{equation}{section}

\newcommand{\K}{{\mathbb{K}}}

\newcommand{\PP}{{\mathbb{P}}}

\newcommand{\Z}{{\mathbb{Z}}}

\newcommand{\ua}{\underline{a}}
\newcommand{\ub}{\underline{b}}

\def\cS{{\mathcal{S}}}
\def\cR{{\mathcal{R}}}
\def\cI{{\mathcal{I}}}
\def\cJ{{\mathcal{J}}}

\def\cK{{\mathcal{K}}}

\def\T{{\mathbf{T}}}
\def\t{{\mathbf{t}}}

\def\X{{\mathbf{X}}}

\def\bft{{\mathbf{t}}}

\def\bff{{\mathbf{f}}}
\def\uv{{\underline{v}}}

\def\Y{{\mathbf{Y}}}
\def\uZ{{\mathbf{Z}}}
\def\uv{{\underline{v}}}

\newcommand{\longhookrightarrow}{\lhook\joinrel\longrightarrow}
\SelectTips{cm}{}

\begin{document}

\title{The Rees Algebra of Parametric Curves via liftings}

\begin{abstract}
We study the defining equations of the Rees algebra of ideals arising from curve parametrizations in the plane and in rational normal scrolls, inspired by the work of Madsen and Kustin, Polini and Ulrich.  The curves are related by work of Bernardi, Gimigliano, and Id\'a, and we use this framework to relate the defining equations.
\end{abstract}

\author{Teresa Cortadellas Ben\'itez}
\address{Universitat de Barcelona, Facultat de Educaci\'o.
Passeig de la Vall d'Hebron 171,
08035 Barcelona, Spain}
\email{terecortadellas@ub.edu}

\author{David A. Cox}
\address{Department of Mathematics \& Statistics, Amherst College, Amherst, MA 01002, USA}
\email{dacox@amherst.edu}

\author{Carlos D'Andrea}
\address{Universitat de Barcelona, Departament de Matem\`atiques i Inform\`atica,
 Universitat de Barcelona (UB),
 Gran Via de les Corts Catalanes 585,
 08007 Barcelona,
 Spain} \email{cdandrea@ub.edu}
\urladdr{http://www.ub.edu/arcades/cdandrea.html}


\date{\today}

\maketitle

\section*{Introduction}
The method of implicitization via moving hypersurfaces of rational parameterized varieties developed by Sederberg and his collaborators in the 90's (cf. \cite{SGH97, CGZ00} and the references therein) can be properly formulated and studied via the Rees algebra of the input data, as shown in \cite{Cox}. Since then, the defining equations of Rees algebras of parametric curves and surfaces have become an active area of research, see for instance \cite{CHW,B09,KPU11,CD13,CD2,KPU13,LP14,CD15,Madsen}.  

In this paper, we study the defining equations of the Rees algebra of ideals arising from curve parametrizations in the plane and rational normal scrolls, and connections between them.  The paper \cite{BGI} by Bernardi, Gimigliano, and Id\'a studies this connection from a \emph{geometric} point of view, while the papers \cite{KPU11} by Kustin, Polini, and Ulrich and \cite{Madsen} by Madsen are more \emph{algebraic}.  Our goal is to link these two approaches.

In more detail, consider a map $\PP^1 \to \PP^2$ defined by $f_{0,d},f_{1,d},f_{2,d}\in\K[T_0,T_1]$ relatively prime of degree $d$.  The syzygy module of $f_{0,d},f_{1,d},f_{2,d}$ has a basis $p,q$ of degrees $\mu \le d-\mu$.  If we write $p = (p_{0,\mu},p_{1,\mu},p_{2,\mu})$, then $p$ has its own syzygy module with generators of degrees $0\leq\mu_1\leq\mu_2$ with $\mu = \mu_1+\mu_2$.  As shown in \cite{BGI}, this leads to a factorization of $\PP^1 \to \PP^2$ into maps
\[
\PP^1 \longrightarrow \cS_{\mu_1,\mu_2} \longhookrightarrow \PP^{\mu+1} \dashrightarrow \PP^2,
\]
where $\cS_{\mu_1,\mu_2} \subseteq \PP^{\mu+1}$ is a rational normal scoll and the final map is a linear projection.  The geometry of this factorization is described in \cite[Thm.\ 3.1]{BGI}.

Our approach is to assemble these maps into the commutative diagram
\begin{equation}
\label{BGIdiagram}
\begin{array}{c}
\SelectTips{cm}{}
\xymatrix@C=30pt{
\PP^1 \ar[drr]^\gamma \ar[dr]^(.65){\gamma_0} \ar[ddr]_\bff & & \\
& \cS_{\mu_1,\mu_2} \ar@{^{(}->}[r] \ar[d] & \PP^{\mu+1} \ar@{-->}[dl] \\
& \PP^2 & }
\end{array}
\end{equation}
where $\gamma, \, \gamma_0,$ and $\bff$ are defined in \eqref{gamma}, \eqref{gamma0}, and \eqref{f} respectively.
We see below that these three maps lead to ideals $I,J,K \subseteq \K[T_0,T_1]$ whose Rees algebras $\cR(I),\cR(J),\cR(K)$ have defining ideals $\cI, \cJ, \cK$.  In Lemma~\ref{maplemma}, we consider a  Rees dual version of  \eqref{BGIdiagram}, which is the commutative diagram:
\begin{equation}
\label{Reesdiagram}
\begin{array}{c}
\SelectTips{cm}{}
\xymatrix{
& \K[T_0,T_1,\X,\Y] \ar[d]^{\Phi'} \ar[ddr]^\Phi & \\
\K[T_0,T_1,\uZ] \ar[ur]^\Gamma \ar[r]^\Omega \ar[rrd]_\psi & \K[T_0,T_1,X,Y] \ar[dr]^\phi & \\
&& \K[T_0,T_1,s]}
\end{array}
\end{equation}
where $\Phi$ comes from $\gamma,\, \phi$ from $\gamma_0,$ and $\psi$ from $\bff.$  The maps $\phi, \psi, \Phi, \Phi',\Gamma,$ and  $\Omega$ in \eqref{Reesdiagram} are defined in \eqref{phi}, \eqref{psidef}, \eqref{phii}, \eqref{phiip}, \eqref{Gama}, and \eqref{Omega} respectively, where the notation $\X, \Y, \uZ$ is also explained.  The connection between these diagrams is the following:
\begin{align*}
&\text{the curve }\, \gamma(\PP^1)\subset \PP^{\mu+1}\, \text{ gives }
\cR(I) = \mathrm{im}(\Phi) \text{ and } \cI = \ker(\Phi) \\[-2pt]
&\text{the curve }\, \gamma_0(\PP^1)\subset\cS_{\mu_1,\mu_2}\, \text{gives } \cR(J) = \mathrm{im}(\phi) \text{ and } \cJ = \ker(\phi) \\[-2pt]
&\text{the curve }\,\bff(\PP^1)\subset \ \PP^{2}\ \ \text{ gives }\cR(K) = \mathrm{im}(\psi) \text{ and } \cK = \ker(\psi).
\end{align*}

The easiest Rees algebra is $\cR(J)$ coming from $\gamma_0.$  In Section~\ref{1}, we show that  the defining ideal $\cJ$ of $\cR(J)$ is especially simple with a nice toric interpretation (Proposition~\ref{alphabetaRees}).    In Section~\ref{2}, we shift to $\gamma,$ which leads to the Rees algebra $\cR(I)$ discussed in \cite{KPU11}.  We explictly describe the minimal generators of $\cI$ in Theorem~\ref{mtm}. Section~\ref{3} explains how our results relate to the papers  \cite{KPU11, LP14, Madsen}.

In Section~\ref{4}, we bring $\bff$ into the picture and explain the diagrams \eqref{BGIdiagram} and \eqref{Reesdiagram} in detail.
Here, the ideal is
\[
K = \langle f_{0,d},f_{1,d},f_{2,d}\rangle \subseteq \K[T_0,T_1],
\]
and as noted above, describing the ideal $\cK$ of defining equations of the Rees algebra $\cR(K)$ is a major unsolved problem. When we present the Rees algebra of $K$ as
\[
\cR(K) = \K[T_0,T_1,Z_0,Z_1,Z_2]/\cK,
\]
the syzygy $p$ gives $p = p_{0,\mu} Z_0 + p_{1,\mu}Z_1 + p_{2,\mu}Z_2 \in \cK$.  If we do the same for the other syzygy $q$, then $p$ and $q$ become part of a minimal generating set of $\cK$.  In Section~\ref{5}, we construct operators $D_A$ and $D_B$ which, when applied successively to $q$, give further minimal generators of $\cK$ (Theorem~\ref{DD}).  In Section~\ref{6} we discuss how our results relate to Madsen's paper \cite{Madsen}, and in Section~\ref{7}, we explain how the minimal generators of $\cK$ constructed in Theorem~\ref{DD} relate to the minimal generators of $\cI$ described earlier in Theorem~\ref{mtm}.

One notational convention is that a second subscript often denotes degree.  We used this above when three polynomials of degree $d$ were denoted $f_{i,d}$ for $i=0,1,2$.

\bigskip

{\bf Acknowledgements}
We are grateful to Kuei-Nan Lin and Yi-Huang Shen for their careful reading of the manuscript and helpful suggestions for improvement. We are also grateful to the anonymous referees for their several suggestions for improving the presentation of the manuscript.
T.~Cortadellas and C.~D'Andrea are both supported by Spanish MINECO/FEDER research projects MTM2013-40775-P, and   MTM 2015-65361-P respectively.  C.~D'Andrea is also supported by the ``Mar\'ia de Maeztu'' Programme for Units of Excellence in R\&D (MDM-2014-0445), and by the European Union's Horizon 2020 research and innovation programme under the Marie Sklodowska-Curie grant agreement No. 675789.

\section{Parametrizations and Toric Surfaces}\label{1}

Assume we have $(d,\,\mu_1,\mu_2)\in\Z^3$ with $0\leq\mu_1\leq\mu_2$ and set $\mu=\mu_1+\mu_2\leq\frac{d}{2}$.  Let $\K$ be a field
and $T_0,T_1$ be variables.  For homogeneous elements $\alpha_{d-\mu_1},\,\beta_{d-\mu_2}\in\K[T_0,T_1]$ of respective degrees $d-\mu_1,\,d-\mu_2$ and no common factors, consider the rational map
\begin{equation}
\label{gamma}
\begin{array}{cccc}
\gamma:&\PP^1&\longrightarrow&\PP^{\mu+1}\\
&\t:=(t_0:t_1)&\longmapsto& \big(t_0^{\mu_1}\alpha_{d-\mu_1}(\t):\dotsb:t_1^{\mu_1}\alpha_{d-\mu_1}(\t):t_0^{\mu_2}\beta_{d-\mu_2}(\t):\dotsb:t_1^{\mu_2}\beta_{d-\mu_2}(\t)\big).
\end{array}
\end{equation}
This is one of the maps appearing in \eqref{BGIdiagram}.  The image of $\gamma$ is a curve lying inside the rational normal surface  $\cS_{\mu_1,\mu_2} \subseteq \PP^{\mu+1}$ defined by
\[
\cS_{\mu_1,\mu_2}=\{(s_0 t_0^{\mu_1}:\dotsb:s_0 t_1^{\mu_1}:s_1 t_0^{\mu_2}:\dotsb:s_1 t_1^{\mu_2}),\, (t_0:t_1),\,(s_0:s_1)\in\PP^1\}.
\]

To approach these objects from a toric point of view, let $X, Y$ be new variables and consider the lattice polygon $P$ with facet variables $T_0,T_1,X,Y$ shown in Figure~\ref{fig1}.
\begin{figure}[h]
\[
\begin{picture}(210,60)\label{fig1}
\thicklines
\put(15,15){\line(1,0){120}}
\put(15,15){\line(0,1){30}}
\put(15,45){\line(1,0){180}}
\put(135,15){\line(2,1){60}}
\multiput(15,15)(30,0){5}{\circle*{3}}
\multiput(15,45)(30,0){7}{\circle*{3}}
\put(0,26){$T_1$}
\put(172,20){$T_0$}
\put(85,2){$Y$}
\put(85,50){$X$}
\put(123,2){$(\mu_1,0)$}
\put(182,51){$(\mu_2,1)$}
\put(4,2){$(0,0)$}
\put(4,51){$(0,1)$}
\end{picture}
\]
\caption{The Lattice Polygon $P$}
\label{polygonfig}
\end{figure}
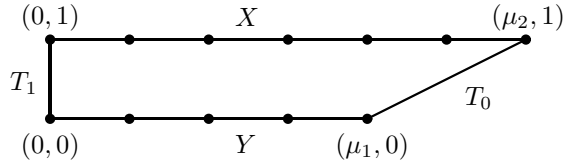

The lattice points in $P$ give the monomials
\begin{equation}
\label{monomials}
\begin{aligned}
&T_0^{\mu_2} Y \ \ T_0^{\mu_2-1}T_1 Y \ \ \cdots \ \ T_1^{\mu_2} Y\\[5pt]
&T_0^{\mu_1} X  \ \ \cdots \ \ T_1^{\mu_1} X
\end{aligned}
\end{equation}
where the exponents are the lattice distances to the edges.  When  we assign toric bidegrees
\begin{equation}
\label{toricgrading}
\deg(T_0) = \deg(T_1) = (1,0),\ \deg(X) = (-\mu_1,1), \ \deg(Y) = (-\mu_2,1),
\end{equation}
the monomials in \eqref{monomials} all have toric bidegree $(0,1)$.  Furthermore, $P$ gives the toric variety $X_P$ which is the Hirzebruch surface $\mathbb{F}_{\mu_2-\mu_1}$, and
$\K[T_0,T_1,X,Y]$, with the toric bigrading given in \eqref{toricgrading}, is the total coordinate ring (Cox ring) of $X_P$ after picking a suitable basis of the Picard group.
The toric geometry used here is explained in
\cite[Chapter 5]{CLS}.

For the above lattice polygon $P$, $X_P$ maps isomorphically to the normal rational surface $\cS_{\mu_1,\mu_2} \subseteq \PP^{\mu+1}$ via the monomials \eqref{monomials}.  Because of this, we identify $X_P$ with its image in $\PP^{\mu+1}$ and write $X_P = \cS_{\mu_1,\mu_2}$.

The image of $\gamma$ lies in $\cS_{\mu_1,\mu_2} \subseteq \PP^{\mu+1}$ and is defined by the equation
\[
\alpha_{d-\mu_1} Y = \beta_{d-\mu_2} X
\]
in the total coordinate ring $\K[T_0,T_1,X,Y]$.  Thus we have the factorization
\begin{equation}\label{gamma0}
\gamma:\PP^1\stackrel{\gamma_0}{\longrightarrow} \cS_{\mu_1,\mu_2} \stackrel{\gamma_1}{\longrightarrow}\PP^{\mu+1},
\end{equation}
with $\gamma_0(\PP^1)\subseteq \cS_{\mu_1,\mu_2}$ defined by $\alpha_{d-\mu_1} Y - \beta_{d-\mu_2} X$.   Note that $\gamma_0$ also appears in the diagram \eqref{BGIdiagram}.

The Rees algebra $\cR(J)$ of $J = \langle \alpha_{d-\mu_1},\beta_{d-\mu_2}\rangle \subseteq \K[T_0,T_1]$ is presented by the map
\begin{equation}
\label{phi}
\begin{array}{rcl}
\phi:\K[T_0,T_1,X,Y]&\longrightarrow&\K[T_0,T_1,s]\\
T_i&\longmapsto& T_i, \ i=0,1\\
 X&\longmapsto& \alpha_{d-\mu_1}\,s\\
 Y&\longmapsto& \beta_{d-\mu_2}\,s.
\end{array}
\end{equation}
This is the map $\phi$ in \eqref{Reesdiagram}. The image of $\phi$ is the Rees algebra $\cR(J)$, and its kernel $\cJ = \ker(\phi) \subseteq \K[T_0,T_1,X,Y]$ gives the defining equations of $\cR(J)$.
The ring $\K[T_0,T_1,X,Y]$ is the total coordinate ring of $\cS_{\mu_1,\mu_2}$, and the ideal $\cJ$ is easy to describe.

\begin{proposition}
\label{alphabetaRees}
The ideal 
$\cJ$ is principal, generated by $\alpha_{d-\mu_1} Y - \beta_{d-\mu_2}  X.$
\end{proposition}

\begin{proof}
Since $\gcd(\alpha_{d-\mu_1},\beta_{d-\mu_2}) = 1$, $J$ is generated by a regular sequence.  This has two consequences:
the natural map $\mathrm{Sym}(J) \to \cR(J)$ is an isomorphism by \cite[p.\ 29]{Vasconcelos}, and 
the syzygy module of $J$ is generated by $(-\beta_{d-\mu_2},\alpha_{d-\mu_1})$, so that
\[
\mathrm{Sym}(J) \simeq \K[T_0,T_1,X,Y]/\langle \alpha_{d-\mu_1} Y - \beta_{d-\mu_2} X\rangle,
\]
which proves the claim.
\end{proof}

 \section{The Rees Algebra of the Space Curve}\label{2}

The map $\gamma_0:\PP^1 \to \cS_{\mu_1,\mu_2}$ gives the easy Rees algebra described in Proposition~\ref{alphabetaRees}.  Combining this with $\cS_{\mu_1,\mu_2} \subseteq \PP^{\mu+1}$ gives the space curve $\gamma: \PP^1 \to \PP^{\mu+1}.$  Here, the Rees algebra is more complicated.

We introduce some notation. Let $s,\,X_0,\dotsb, X_{\mu_1},\,Y_0,\dotsb, Y_{\mu_2}$ be new variables. We set $\X=X_0,\dotsb, X_{\mu_1},\ \Y= Y_0,\dotsb, Y_{\mu_2}$ 
and $\T=T_0,T_1$  for short. For any $\ell \ge 1$, we also set $\T^{\ell}=T_0^{\ell}, T_0^{\ell-1}T_1,\dotsb, T_1^{\ell}$. With this notation, the map \eqref{gamma} is written more compactly as $\gamma = (\alpha_{d-\mu_1} \T^{\mu_1} : \beta_{d-\mu_2} \T^{\mu_2})$, and the entries of $\gamma$ give the ideal $I = \langle \alpha_{d-\mu_1} \T^{\mu_1},\beta_{d-\mu_2} \T^{\mu_2}\rangle \subseteq \K[\T]$.  The Rees algebra $\cR(I)$ is presented by the map
\begin{equation}\label{phii}
\begin{array}{cccl}
\Phi:&\K[\T,\,\X,\,\Y]&\longrightarrow&\K[\T,s]\\
&T_i&\longmapsto& T_i, \ \  \ i=0,\,1\\
& X_i&\longmapsto& \alpha_{d-\mu_1}\,T_0^{\mu_1-i}T_1^{i}\,s,\ 0\leq i\leq\mu_1\\
& Y_i&\longmapsto& \beta_{d-\mu_2}\,T_0^{\mu_2-i}T_1^{i}\,s,\ 0\leq i\leq \mu_2.
\end{array}
\end{equation}
This is $\Phi$ in \eqref{Reesdiagram}.  As noted above, $\cR(I) = \mathrm{im}(\Phi)$, and $\cI = \ker(\Phi) \subseteq\K[\T,\X,\Y]$ gives the defining equations of the Rees algebra.

Consider also the map
\begin{equation}
\label{phiip}
\begin{array}{rcl}
\Phi':\K[\T,\X,\Y]&\longrightarrow&\K[\T,X,Y]\\
T_i&\longmapsto& T_i,\ i=0,\,1\\
 X_i&\longmapsto& T_0^{\mu_1-i}T_1^{i}\,X,\ i = 0,\dotsb,\mu_1\\
 Y_i&\longmapsto& T_0^{\mu_2-i}T_1^{i}\,Y,\ i = 0,\dotsb, \mu_2,
\end{array}
\end{equation}
and denote with $\cI'$ its kernel.  The
     variables $X_0, X_2, \dotsc, X_{\mu_1}, Y_{0}, Y_{1}, \dotsc, Y_{\mu_2}$ map to the monomials in \eqref{monomials}.  Observe that $\Phi'$ appears in \eqref{Reesdiagram} and corresponds to the inclusion $\cS_{\mu_1,\mu_2} \subseteq \PP^{\mu+1}$ in \eqref{BGIdiagram}.

The rings $\K[\T,\X,\Y]$ and $\K[\T,s]$ have bigradings defined by $\deg(T_i) = (1,0)$, $\deg(X_i) = \deg(Y_i) = (0,1)$ and $\deg(s) = (-d,1)$, and $\K[\T,X,Y]$ has the bigrading defined in \eqref{toricgrading}.  The maps $\Phi$, $\Phi'$ and $\phi$ all preserve these bigradings.

\begin{theorem} \
\label{phiipproperties}
\begin{enumerate}
\item The ideal $\cI$ is the inverse image of $\cJ = \langle \alpha_{d-\mu_1} Y - \beta_{d-\mu_2}  X\rangle$ via  $\Phi'$.
\item The ideal $\cI' $ is generated by all  $F \in \cI$  which are  $(\X,\Y)$-bihomogeneous.
\end{enumerate}
\end{theorem}

\begin{proof}
For part (1), observe that $\Phi = \phi \circ \Phi'$, so that
\[
\cI =  \ker(\Phi) = \Phi'{}^{-1}(\ker(\phi)).
\]
Since $\ker(\phi) = \langle \alpha_{d-\mu_1} Y - \beta_{d-\mu_2}  X\rangle$ by Proposition~\ref{alphabetaRees}, part (1) follows immediately.

For part (2), take $F\in \ker(\Phi')$ and write $F = \sum_{j,k} F_{j,k}$, where the polynomial $F_{j,k}$ is $(\X,\Y)$-bihomogeneous of bidegree $(j,k)$.  Then
\[
0 = F(\T,X \T^{\mu_1},Y \T^{\mu_2}) = \sum_{j,k} F_{j,k}(\T,X \T^{\mu_1},Y \T^{\mu_2}) = \sum_{j,k} X^j Y^k F_{j,k}(\T,\T^{\mu_1},\T^{\mu_2}),
\]
which implies that $F_{j,k}(\T,\T^{\mu_1},\T^{\mu_2}) = 0$ for all $j,k$.  Using the homogeneity again, we see that $F_{j,k}(\T,X\T^{\mu_1},Y\T^{\mu_2}) = 0$, so that $F_{j,k} \in \ker(\Phi') \subseteq \ker(\Phi) = \cI$.  Thus $F_{j,k}$ is a $(\X,\Y)$-bihomogeneous element of $\cI$.  By \eqref{phiip}, we conclude that $F_{j,k}$ and hence $F$ lie in the ideal generated by $(\X,\Y)$-bihomogeneous elements of $\cI$.

For the opposite inclusion, we show that if $F \in \cI$ is $(\X,\Y)$-bihomogeneous of bidegree $(j,k)$, then $F\in \ker(\Phi')$.  To see why, note that by part~(1), $F\in \cI$ implies that
\[
F(\T,X \T^{\mu_1},Y\T^{\mu_2}) = (\alpha_{d-\mu_1} Y - \beta_{d-\mu_2} X) H
\]
for some $H \in \K[\T,X,Y]$.  But being $(\X,\Y)$-bihomogeneous of bidegree $(j,k)$ implies that $F(\T,X \T^{\mu_1},Y\T^{\mu_2}) = X^j Y^k F(\T,\T^{\mu_1},\T^{\mu_2})$, so that
\begin{equation}
\label{eqinKTX0Y0}
X^j Y^k F(\T,\T^{\mu_1},\T^{\mu_2}) = (\alpha_{d-\mu_1} Y - \beta_{d-\mu_2} X) H.
\end{equation}
However, $\alpha_{d-\mu_1}$ and $\beta_{d-\mu_2}$ are nonzero and relatively prime, so that $\alpha_{d-\mu_1} Y - \beta_{d-\mu_2} X$ is irreducible in $\K[\T,X,Y]$ and hence the only way it can divide the left-hand side of \eqref{eqinKTX0Y0} is for the left-hand side to vanish.  But when this happens, we get  $F\in \ker(\Phi')$.
\end{proof}

\begin{proposition}\label{2.2} The ideal
$\cI'$ is minimally generated by
 the pencils
\begin{equation}\label{pencils}T_1X_{i-1}-T_0X_i,\, T_1Y_{j-1}-T_0Y_j, \ 1\leq i\leq\mu_1,\,1\leq j\leq\mu_2,
\end{equation}  and
 the quadrics
\begin{itemize}
\item[] $X_iX_j-X_{i-1}X_{j+1}, 1\leq i\leq j\leq\mu_1-1$,
\item[] $Y_iY_j-Y_{i-1}Y_{j+1}, 1\leq i\leq j\leq\mu_2-1$,
\item[] $X_iY_j-X_{i-1}Y_{j+1},\, 1\leq i\leq\mu_1,\,0\leq j\leq\mu_2-1$.
\end{itemize}
Moreover, this family  is a minimal Gr\"obner basis of $\cI'$ for the monomial order  grevlex $T_0\succ T_1\succ X_0\succ
\dotsb\succ X_{\mu_1}\succ Y_0\succ\dotsb\succ Y_{\mu_2}$.
\end{proposition}
\begin{remark}\label{numm}
The number of elements of the family of minimal generators given in Proposition \ref{2.2} is equal to
$\mu_1+\mu_2+\binom{\mu_1}{2}+\binom{\mu_2}{2}+\mu_1\mu_2$.
\end{remark}
\begin{proof}[Proof of Proposition \ref{2.2}]
Let us show first that those binomials above are a Gr\"obner basis of $\cI'$ for the monomial order stated above. The leading terms of this family are the following:
\begin{itemize}
\item[] $T_1X_{i},\, T_1Y_{j}, \ 0\leq i\leq\mu_1-1,\,0\leq j\leq\mu_2-1$,
\item[] $X_iX_j,\, 1\leq i\leq j\leq\mu_1-1$,
\item[] $Y_iY_j,\, 1\leq i\leq j\leq\mu_2-1$,
\item[] $X_iY_j,\, 1\leq i\leq\mu_1,\,0\leq j\leq\mu_2-1$.
\end{itemize}

Due to \eqref{phiip}, we deduce straightforwardly that $\cI'$ is a \textit{trihomogeneous} ideal in the groups of variables $(\T,\X,\Y)$, so it is enough to test the membership of this kind of elements.
In what follows, we refer to such a polynomial as trihomogeneous, or if we want to specify the degrees, as $(i,j,k)$-homogeneous.

As $\cI'$ is a prime ideal, a minimal set of generators consists of a system of irreducible elements. Given a nonzero irreducible $(i,j,k)-$trihomogeneous element $F_{i,j,k}$, if its leading monomial is not divisible by any of the leading terms of the binomials in the family above, then it must be of one of the following forms:
\begin{enumerate}
\item $T_0^i X_\ell,\,1\leq \ell\leq \mu_1$ (so $j=1,\,k=0$),
\item $T_0^iY_\ell,\,1\leq \ell\leq \mu_2$ (so $j=0,\,k=1$),
\item  $T_1^i X_{\mu_1}^jY_{\mu_2}^k$,
\item $X_0^{j'}X_\ell^{\{0,1\}}X_{\mu_1}^{j''}Y_{\mu_2}^k,\, 1 \leq \ell \leq \mu_1-1$ (so $i=0$),
\item $X_{\mu_1}^{j}Y_0^{k'}Y_{\ell}^{\{0,1\}}Y_{\mu_2}^{k''},\, 1 \leq \ell \leq \mu_2-1$ (so $i=0$),
\end{enumerate}
with $j'+j''\in\{j, j-1\},\,k'+k''\in\{k,k-1\}$. Here, $X_\ell^{\{0,1\}}$ means that there are only two possible exponents for $X_\ell: 0$ or $1$.
We deal with each of these cases:
\begin{enumerate}
\item Any other monomial appearing in the expansion of $F_{i,1,0}$ must be of the form
$T_0^{i'}T_1^{i''}X_{\ell'}$ with $i'+i''=i,\,\ell'>\ell$. After the specialization given by \eqref{phiip}, we get that $T_0^i X_\ell$ gets converted into $T_0^{i+\mu_1-\ell}T_1^{\ell}s$, while $T_0^{i'}T_1^{i''}X_{\ell'}$
maps to $T_0^{i'+\mu_1-\ell'}T_1^{i''+\ell'}s$. We have $i''+\ell'>\ell$, so the image of the leading term cannot be cancelled, which shows that such a polynomial cannot be in the kernel.
\item The same argument used in (1) applies here.
\item After specializing the polynomial with \eqref{phiip}, we get that
$T_1^iX_{\mu_1}^jY_{\mu_2}^k$ maps into $T_1^{i+j\mu_1+k\mu_2}s^{j+k}$, and any other nonzero term of $F_{i,j,k}$ is converted into a multiple of $T_0$. So, $F_{i,j,k}$ cannot be in the kernel of $\Phi'$, and hence this monomial cannot be  the leading monomial of any element of $\cI'$.

\item As $F_{0,j,k}$ is trihomogenous, due to the way we defined the monomial order, any other monomial in the expansion of $F_{i,j,k}$ must be a multiple of $Y_{\mu_2}^k$. As we assumed this polynomial irreducible, this forces $k=0$, and in fact the leading monomial is $X_0^{j'}X_\ell^{j'''}X_{\mu_1}^{j''}$, with $j'''\in\{0,1\}$. Applying \eqref{phiip}, it becomes $T_0^{\mu_1j'+(\mu_1-\ell)j'''}T_1^{\ell j'''+\mu_1j''}s^j$. As before, any other monomial in $F_{0,j,0}$ maps to a strictly larger power of $T_1$, hence the specialized polynomial cannot be identically zero. This shows that no element in $\cI'$ can have this leading term.

\item  As $X_{\mu_1}Y_\ell$ is one the leading terms of the quadrics in the statement of the claim, we have that if $Y_\ell$ actually appears in the monomial, then $j=0$, and this case can be solved like in (4). Suppose then that this is not the case. The leading monomial then turns into  $X_{\mu_1}^{j}Y_0^{k'}Y_{\mu_2}^{k''}$. As $X_{\mu_1}Y_0$ is also one of the leading terms of the quadrics above, we now have that either $j=0$ or $k'=0$. The case $j=0$ gets solved as before, and in the other one, we get that the leading monomial actually is
$X_{\mu_1}^{j}Y_{\mu_2}^k$, which is the case we have dealt with in (2).
\end{enumerate}

So, we get that the family of elements in the claim is a Gr\"obner basis of $\cI'$. In particular, they generate this ideal. It is easy to see that it is a minimal Gr\"obner basis, as the leading terms have all total degree $2$ and they are pairwise different. To show that it is also a minimal set of generators, note that all of them have total degree $2$, and hence if one of these binomials is a combination of the others, it must be a $\K$-linear combination of them. Choose the polynomial in this nontrivial linear combination with the highest leading term among all the polynomials in the combination. This highest leading term cannot be cancelled by any of the other summands, which is a contradiction. So, the family is minimal and this concludes with the proof of the theorem.
\end{proof}

Now we search for trihomogeneous nontrivial elements  of $\cI$.

\begin{lemma}\label{util2p}
If $A_{i,j,k}\in\K[\T,\X,\Y]$ is $(i,j,k)$-trihomogeneous with $k\geq1$,  such that $$A_{i,j,k}(\T,\T^{\mu_1},\T^{\mu_2})=\alpha_{d-\mu_1}\cdot q_{i+(j+1)\mu_1+k\mu_2-d},$$ with
$q_{i+(j+1)\mu_1+k\mu_2-d}\in\K[\T]$ nonzero, homogeneous of degree $i+(j+1)\mu_1+k\mu_2-d$, then there exists $B_{i,j+1,k-1}\in\K[\T,\X,\Y]$ $(i,j+1,k-1)$-trihomogeneous  such that
\begin{equation}\label{oxix}
A_{i,j,k}-B_{i,j+1,k-1}\in\cI.
\end{equation}
\end{lemma}

\begin{proof}
The polynomial $0\neq q_{i+(j+1)\mu_1+k\mu_2-d}\beta_{d-\mu_2}\in\K[\T]$ has degree
$$i+(j+1)\mu_1+k\mu_2-d+d-\mu_2=i+(j+1)\mu_1+(k-1)\mu_2,$$ so we can write it as
$B_{i,j+1,k-1}(\T,\T^{\mu_1},\T^{\mu_2})$ for some $B_{i,j+1,k-1}\in\K[\T,\X,\Y]$ $(i,j+1,k-1)$-trihomogeneous. This polynomial satisfies the claim.
\end{proof}
\begin{remark}
All choices of $B_{i,j+1,k-1}$ in \eqref{oxix} must satisfy that $B_{i,j+1,k-1}(\T,\T^{\mu_1},\T^{\mu_2})$ is equal to a fixed polynomial. Hence, two different choices for this form are equivalent modulo $\cI'$.
\end{remark}

\subsection{Minimal Generators} \label{33}
Now we exhibit a family of minimal generators of $\cI$.
Let $\mathbf{V} = (V_0, V_1, V_2)$ be new variables, set $\mu=\mu_1+\mu_2$, and consider the monomial ideal in $\K[\mathbf{V}]$:
\begin{equation}\label{Smu}
S_{\mu_1,\mu_2,d}=\langle V_0^iV_1^jV_2^k \mid (i,j,k)\ \in(\Z_{\geq0})^3\,\mbox{with}\, i+\mu_1j+\mu_2k\geq d-\mu\rangle.
\end{equation}
By the Hilbert Basis Theorem, $S_{\mu_1,\mu_2,d}$ has a unique minimal set of monomial generators:
$$S_{\mu_1,\mu_2,d}=\langle V_0^{i_1}V_1^{j_1}V_2^{k_1},\dotsb, V_0^{i_N}V_1^{j_N}V_2^{k_N} \rangle.
$$
\begin{remark}\label{rett}
If $d\geq3$, we have that neither $V_0$ nor $V_1$ nor $V_2$ belong to $S_{\mu_1,\mu_2,d}$ as $0\leq\mu_1\leq\mu_2$, and $d-\mu\geq\frac{d}2>1$.
\end{remark}
In the following Lemma, whose proof is straightforward, we make the elements of $S_{\mu_1,\mu_2,d}$ more explicit. 

\begin{lemma}
\label{useful}
The elements of  $S_{\mu_1,\mu_2,d}$ are of the form
$V_0^{d-(a+1)\mu_1-(b+1)\mu_2}V_1^{a}V_2^{b},$ with $a,b \in\Z_{\geq0}, a\mu_1+b\mu_2<d-\mu,
$ or  $V_1^aV_2^b$, with $(a,b)$ in the Hilbert Basis of  $\{(x,y)\in(\Z_{\geq0})^2\mid\mu_1x+\mu_2y\geq d-\mu\}$.
\end{lemma}

Write $\uv_\ell=(i_\ell, j_\ell, k_\ell)$, and set
\begin{equation}\label{sl}
s_\ell=i_\ell+(j_\ell+1)\mu_1+(k_\ell+1)\mu_2-d\geq0.
\end{equation}
The following result is needed to prove Theorem \ref{mtm} below.
\begin{lemma}\label{prima}
For $\ell=1,\dotsb, N$, we have the following:
\begin{enumerate}
\item $0\leq s_\ell<\mu_2$.
\item $s_\ell = 0$ whenever $i_\ell > 0$.
\end{enumerate}
\end{lemma}

\begin{proof}
(1) An exponent $\uv_\ell=(i_\ell,j_\ell,k_\ell)$ appears among the minimal generators of $S_{\mu_1,\mu_2,d}$ if and only if the following three triplets either do not belong to $(\Z_{\geq0})^3$ or the corresponding monomial does not belong to the monomial ideal:
$$(i_\ell-1,j_\ell,k_\ell),\,(i_\ell,j_\ell-1,k_\ell),\,(i_\ell,j_\ell,k_\ell-1).$$  If one or two of the exponents are zero, then we need to consider fewer cases, so w.l.o.g.\ we can assume that  the three of them are positive.
In the first case, we have that $s_\ell=0<\mu_2$, in the second, we get $0\leq s_\ell<\mu_1\leq \mu_2$, and in the third, we have $0\leq s_\ell<\mu_2$.

(2) If $i_\ell > 0$ and $s_\ell > 0$, then $V_0^{i_\ell-1}V_1^{j_\ell}V_2^{k_\ell} \in S_{\mu_1,\mu_2,d}$.  It follows that
$V_0^{i_\ell}V_1^{j_\ell}V_2^{k_\ell} = V_0 \cdot V_0^{i_\ell-1}V_1^{j_\ell}V_2^{k_\ell}$ cannot be a minimal generator.
\end{proof}

For each $\uv_\ell$ and $0\leq t\leq s_\ell$, let
$A^t_{i_\ell,j_\ell,k_\ell+1}\in\K[\T,\X,\Y]$ the tri-homogeneous polynomial such that
\begin{equation}\label{ambi}A^t_{i_\ell,j_\ell,k_\ell+1}(\T,\T^{\mu_1}, \T^{\mu_2})=\alpha_{d-\mu_1} T_0^tT_1^{s_\ell-t}
\end{equation}
(it is easy to see that there always exists such a polynomial, and moreover any two choices for $A^t_{i_\ell,j_\ell,k_{\ell}+1}$ coincide modulo $\cI'$), and set
\begin{equation}\label{cozi}
\Psi^t_{{i_\ell,j_\ell,k_\ell+1}}=A^t_{{i_\ell,j_\ell,k_\ell+1}}-B^t_{i_\ell,j_\ell+1,k_\ell},
\end{equation}
where $B^t_{i_\ell,j_\ell+1,k_\ell}$ has been defined in \eqref{oxix}.

\begin{theorem}\label{mtm}
The ideal  $\cI$ is minimally generated by a set of minimal generators of $\cI'$ plus the family
\begin{equation}\label{flii}
\{\Psi^t_{i_\ell,j_\ell,k_\ell+1}\mid 1\leq \ell\leq N,\, 0\leq t\leq s_\ell\}.
\end{equation}
\end{theorem}

\begin{remark}\label{rkr}
Note that the cardinality of \eqref{flii} is equal to $\sum_{\ell=1}^N(s_\ell+1)$.
\end{remark}

\begin{proof}[Proof of Theorem \ref{mtm}]
Let $g = \alpha_{d-\mu_1} Y - \beta_{d-\mu_2} X \in \K[\T,X,Y]$.  
Theorem~\ref{phiipproperties} implies that $\Phi'(\cI) \subseteq \cJ = \langle \alpha_{d-\mu_1} Y - \beta_{d-\mu_2}  X\rangle = g\K[\T,X,Y]$ for the map $\Phi'$ defined in 
\eqref{phiip}.  Since $\Phi'$ preserves the bigrading and $\Psi^t_{i_\ell,j_\ell,k_\ell+1} \in \cI$ by construction, we have inclusions
\[
\Phi'(\langle \Psi^t_{i_\ell,j_\ell,k_\ell+1}\mid 1\leq \ell\leq N,\, 0\leq t\leq s_\ell\rangle) \subseteq \Phi'(\cI) \subseteq (g\K[\T,X,Y])_{\ge0,*}.
\]
We claim that these inclusions are equalities, i.e., 
\begin{equation}
\label{2.9.claim}
\Phi'(\langle \Psi^t_{i_\ell,j_\ell,k_\ell+1}\mid 1\leq \ell\leq N,\, 0\leq t\leq s_\ell\rangle) =  \Phi'(\cI) = (g \K[\T,X,Y])_{\ge0,*}.
\end{equation} 
To prove this, first recall that $\Psi^t_{{i_\ell,j_\ell,k_\ell+1}}=A^t_{{i_\ell,j_\ell,k_\ell+1}}-B^t_{i_\ell,j_\ell+1,k_\ell}$, where
\begin{align*}
A^t_{i_\ell,j_\ell,k_\ell+1}(\T,\T^{\mu_1}, \T^{\mu_2})&=\alpha_{d-\mu_1} T_0^tT_1^{s_\ell-t}\\
B^t_{i_\ell,j_\ell+1,k_\ell}(\T,\T^{\mu_1}, \T^{\mu_2}) &= \beta_{d-\mu_2} T_0^tT_1^{s_\ell-t}.
\end{align*}
Since $A^t_{i_\ell,j_\ell,k_\ell+1}$ and $B^t_{i_\ell,j_\ell+1,k_\ell}$ are trihomogeneous, it follows easily that
\begin{align*}
\Phi'(\Psi^t_{{i_\ell,j_\ell,k_\ell+1}}) &= \Psi^t_{{i_\ell,j_\ell,k_\ell+1}}(\T,\T^{\mu_1}X, \T^{\mu_2}Y)\\
&= T_0^tT_1^{s_\ell-t}X^{j_\ell} Y^{k_\ell} \big(\alpha_{d-\mu_1} Y - \beta_{d-\mu_2} X) =  T_0^tT_1^{s_\ell-t}X^{j_\ell} Y^{k_\ell} g.
\end{align*}

It suffices to show $(g \K[\T,X,Y])_{\ge0,*} \subseteq \Phi'(\langle \Psi^t_{i_\ell,j_\ell,k_\ell+1}\mid 1\leq \ell\leq N,\, 0\leq t\leq s_\ell\rangle)$. Suppose that $H g \in \big(g\cR(F)\big)_{\ge0,*}$, and let $T_0^u T_1^{s-u} X^j Y^k$ be a monomial appearing in $H$.  Since $\deg(g) = (d-\mu,1),\, \deg(X) = (-\mu_1,1),\, \deg(Y) = (-\mu_2,1)$, we have
\[
\deg(T_0^u T_1^{s-u} X^j Y^k g) = (s-j\mu_1-k\mu_2+d-\mu,j+k+1),
\]
so that $i := s-j\mu_1-k\mu_2+d-\mu \ge 0$.  Then $i+ j\mu_1+k\mu_2-d+\mu = s\ge 0$, which implies that $V_0^i V_1^j V_2^k \in S_{\mu_1,\mu_2,d}$ from \eqref{Smu}.  It follows that for some $1 \le \ell \le N$ and $0 \le t \le s_\ell$, we have
\[
i \ge i_\ell,\ j \ge j_\ell,\ k \ge k_\ell,
\]
from which we conclude $s \ge s_\ell$.  It is then straightforward to find an integer $0 \le t \le s_\ell$ and a monomial $\T^{\underline{u}} \X^{\underline{v}} \Y^{\underline{w}}$ of tridegree $(i-i_\ell,j-j_\ell,k-k_\ell)$ such that
\[
\Phi'(\T^{\underline{u}} \X^{\underline{v}} \Y^{\underline{w}} \,\Psi^t_{{i_\ell,j_\ell,k_\ell+1}}) = T_0^u T_1^{s-u} X^j Y^k g.
\]
It follows that $H g \in \Phi'(\langle \Psi^t_{i_\ell,j_\ell,k_\ell+1}\mid 1\leq \ell\leq N,\, 0\leq t\leq s_\ell\rangle)$, which completes the proof of \eqref{2.9.claim}.

This tells us that the ideals $\langle \Psi^t_{i_\ell,j_\ell,k_\ell+1}\mid 1\leq \ell\leq N,\, 0\leq t\leq s_\ell\rangle$ and $\cI$ have the same image under $\Phi'$.  Since $\cI' = \ker(\Phi')$, we conclude that $\cI$ is generated by $\cI'$ and  $\{\Psi^t_{i_\ell,j_\ell,k_\ell+1}\mid 1\leq \ell\leq N,\, 0\leq t\leq s_\ell\}$.

Let us prove that the family is minimally generated.  
If $d=2,\,\mu_1=0,\,\mu_2=1,$ then we get by direct calculation that $T_1 Y_0 - T_0 Y_1$ generates $\cI'$, and by writing  $\alpha_{d-\mu_1} = \alpha_2 = aT_0^2 + bT_0T_1 + cT_1^2$ and $\beta_{d-\mu_2} = \beta_1,$ we can express the elements in \eqref{flii} as
$$\Psi^0_{1,0,1} = A^0_{1,0,1} - B^0_{1,1,0} = (aT_0 Y_0 + bT_0Y_1 + cT_1 Y_1) - \beta_1(T_0,T_1) X_0$$ and 
$ \Psi^0_{0,0,2} = A^0_{0,0,2} - B^0_{0,1,1} = \alpha_2(Y_0,Y_1) - \beta_1(Y_0,Y_1) X_0.$ It is easy to see that these three elements form a minimal set of generators of $\cI.$

The remaining cases are $\mu_1=\mu_2=0$ or $d\geq3$. In the first case,  a direct computation shows that $\cI'=0.$ In the latter, thanks to Remark \ref{rett} we see that the only elements in the family of generators of total degree two are the generators of $\cI'$,  which is a minimal generating set of $\cI'$ by Proposition \ref{2.2}. 
Suppose then that an element of \eqref{flii} can be written as a polynomial combination of the others modulo $\cI'$, i.e., 
\[
\Psi^{t_0}_{{i_{\ell_0},j_{\ell_0},k_{\ell_0}+1}}-\sum Q^{t,\ell} \hskip1pt \Psi^t_{{i_\ell,j_\ell,k_\ell+1}}\in\cI'.
\]
Applying $\Phi'$, we obtain
\begin{equation}
\label{mostro}
T_0^{t_0} T_1^{s_{\ell_0} - t_0} X^{j_{\ell_0}} Y^{k_{\ell_0}} g = \sum \Phi'(Q^{t,\ell})  \hskip1pt T_0^{t} T_1^{s_{\ell} - t} X^{j_{\ell}} Y^{k_{\ell}} g
\end{equation}
in $\K[\T,X,Y]$. Consider the map $\pi : \K[\T,X,Y] \to \K[V_0^{\pm1},V_1,V_2]$ defined by 
\[
(T_0,T_1,X,Y) \longmapsto (V_0,V_0,V_0^{-\mu_1}V_1,V_0^{-\mu_2}V_2).
\]
If we divide \eqref{mostro} by $g$ and apply $\pi$, we obtain an equation
\[
V_0^{s_{\ell_0} - \mu_1 j_{\ell_0} - \mu_2 k_{\ell_0}} V_1^{j_{\ell_0}} V_2^{k_{\ell_0}} =\sum \pi(\Phi'(Q^{t,\ell})) \hskip1pt
V_0^{s_{\ell}- \mu_1 j_{\ell} - \mu_2 k_{\ell}} V_1^{j_{\ell}} V_2^{k_{\ell}}
\]
in $\K[V_0^{\pm1},V_1,V_2]$.  However, one checks that $\pi(\Phi'(X_i)) = \pi(T_0^{\mu_1-i}T_1^i X) = V_1$, and similarly $\pi(\Phi'(Y_i)) = V_2$.  It follows that $\pi(\Phi'(Q^{t,\ell}))$ is a polynomial in $\mathbf{V} = (V_0,V_1, V_2)$.  Hence, if we multiply each side by $V_0^{d-\mu}$, we obtain the equation
\begin{equation}
\label{mostro2}
V_0^{i_{\ell_0}} V_1^{j_{\ell_0}} V_2^{k_{\ell_0}} =\sum \pi(\Phi'(Q^{t,\ell}))  \hskip1pt
V_0^{i_{\ell}} V_1^{j_{\ell}} V_2^{k_{\ell}}
\end{equation}
in $\K[\mathbf{V}]$.   This is impossible since the monomials \eqref{mostro2} are minimal generators of the ideal $S_{\mu_1,\mu_2,d} \subseteq \K[\mathbf{V}]$ defined in \eqref{Smu}.
\end{proof}

 \begin{example}
Set $\mu_1=3,\,\mu_2=5$, and $d=17$.
The exponents of the minimal generators of the monomial ideal $S_{3,5,17}$ can be easily computed to be
$$\{(9,0,0),\,(0,3,0),\,(0,0,2),\,(1,1,1),\,(0,2,1),\,(4,0,1),\,(3,2,0),\,(6,1,0)\}.$$

The number $s_\ell$ defined in \eqref{sl} is always equal to $0$ except for $(0,2,1)$, where it is $2$, and for $(0,0,2)$, where it is  $1$. Due to Remark \ref{rkr}, the family of minimal generators \eqref{flii}  of $\cI$ modulo $\cI'$ has then cardinality $11$. In addition,  thanks to Remark \ref{numm}, we know that $\cI'$ has $36$ minimal generators.
\par To confirm all these numbers with a computational example, we set
$\alpha_{d-\mu_1}=T_0^{14}$, and $\beta_{d-\mu_2}=T_1^{12}$. An explicit computation with \emph{Macaulay2} (\cite{mac2}) gives the following set of generators of $\cI'$:
 \begin{itemize}
\item[]
$T_1Y_4  - T_0Y_5 , T_1Y_3  - T_0Y_4 , T_1Y_2  - T_0Y_3 , T_1Y_1  - T_0Y_2 , T_1Y_0-T_0Y_1$,
\item[]
$T_1X_2-T_0X_3, T_1X_1-T_0X_2, T_1X_0-T_0X_1$,
\item[] $Y_4^2-Y_3Y_5,\,Y_4Y_3-Y_2Y_5,\,Y_4Y_2-Y_1Y_5,Y_1Y_4-Y_0Y_5,\,Y_3^2-Y_1Y_5,\,\\Y_3Y_2-Y_0Y_5,\,Y_1Y_3-Y_0Y_4,\,Y_2^2-Y_0Y_4,\,Y_1Y_2-Y_0Y_3,\,Y_1^2-Y_0Y_2$,
\item[] $X_3Y_4-X_2Y_5, \,X_2Y_4-X_1Y_5,\, X_1Y_4-X_0Y_5,\,X_3Y_3-X_1Y_5,\,X_2Y_3-X_0Y_5,\,\\
X_1Y_3-X_0Y_4,\,X_3Y_2-X_0Y_5,\,X_2Y_2-X_0Y_4,X_1Y_2-X_0Y_3,\,X_3Y_1-X_0Y_4,\,\\ X_2Y_1-X_0Y_3,\,X_1Y_1-X_0Y_2,\,X_3Y_0-X_0Y_3,\,X_2Y_0-X_0Y_2,\,X_1Y_0-X_0Y_1$,
\item[] $X_2^2-X_1X_3,\,X_1X_2-X_0X_3,\,X_1^2-X_0X_2$.
\end{itemize}
Then the following elements complete a system of minimal generators of $\cI$:
\begin{itemize}
\item[] $Y_0^2Y_1-X_3Y_5^2,\,Y_0^3-X_2Y_5^2,\,  X_0^2Y_0Y_2-X_3^3Y_5,\,X_0^2Y_0Y_1-X_2X_3^2Y_5,\, \\X_0^2Y_0^2-X_1X_3^2Y_5, X_3^4-X_0^3Y_0$,
\item[] $T_0X_0Y_0^2-T_1X_3^2Y_5,\,T_0^4Y_0^2-T_1^4X_0Y_5,\,T_1^3X_3^3-T_0^3X_0^2Y_0,\,T_1^6X_3^2-T_0^6X_0Y_0,\,\\T_1^9X_3-T_0^9Y_0$.
\end{itemize}
 \end{example}

\section{Comparison with Previous Work}\label{3}
Minimal generators for $\cI$ have been previously studied and made explicit by  Kustin,  Polini and Ulrich in  \cite{KPU11}.  One of their main results \cite[Theorem 3.6]{KPU11} states that $\cI$  (${\mathcal A}$  in their paper) is generated by $\cI'$ (their $H$) modulo an ideal generated by eligible tuples which can be seen to be in one-to-one correspondence with our \eqref{flii}. Their elements can be constructed explicitly, see \cite[Definition 3.5]{KPU11}, although in a more complicated way. They indeed claim that in page 25{:}\hskip1pt\emph{``we obtain closed formulas for the defining equations of ${\mathcal R}(I)$ {\rm(}Theorem 3.6\hskip.5pt{\rm)} which turn out to be tremendously complicated despite the seemingly strong assumptions on $I$!''}

The construction of the equations is similar to what we have done above: they start with the forms $\alpha_{d-\mu_1},\,\beta_{d-\mu_2}$, and after some calculations in  \cite[Definition 3.3]{KPU11}, they produce the polynomials that belong to the list of minimal generators (\cite[Definition 3.5]{KPU11}).

The strategy used to prove their result goes as follows.  The quotient ring $A = \K[\T,\X,\Y]/\cI'$ is the coordinate ring of the 3-dimensional rational normal scroll $S_{1,\mu_1,\mu_2} \subseteq \PP^{\mu+3}$ coming from the lattice polytope shown in Figure~\ref{PolytopeOfA}.  These are the ``rational normal scrolls'' in the title of \cite{KPU11}.

\begin{figure}[t]
\[
{\psset{unit=1pt}\begin{pspicture}(210,55)
\psline[linestyle=dashed](0,0)(15,15)(135,15)
\psline(135,15)(195,45)(15,45)(0,0)
\psline[linestyle=dashed](15,15)(15,45)
\psline(0,0)(30,0)(135,15)
\psline(30,0)(195,45)
\rput[bl](61,19){$\mu_1$}
\rput[bl](81,49){$\mu_2$}
\rput[bl](16,3.5){$1$}
\end{pspicture}}
\]
\caption{The Lattice Polytope for $S_{1,\mu_1,\mu_2}$}
\label{PolytopeOfA}
\end{figure}
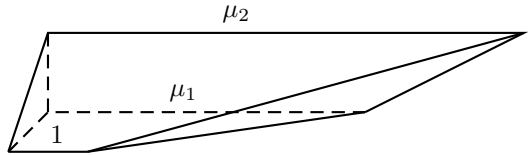

The inclusion  $\cI' \subseteq \cI$ induces a surjection
\begin{equation}
\label{AtoRees}
A \twoheadrightarrow \cR(I).
\end{equation}
As noted in \cite{KPU11}, the ring $A$ gives a better approximation to the  Rees algebra $\cR(I)$ than the symmetric algebra of $I$.  In fact, $\cR(I) = A/\cI A$, where $\cI A$ is a height one prime ideal.  In Theorem 1.11, a monomial ideal $K \subseteq A$ is defined such that $K^{(d-\mu)}$, its  $(d-\mu)$-th symbolic power, is isomorphic to $\cI A$, keeping the bigrading. Then  the minimal (monomial) generators of $K^{(d-\mu)}$ are made explicit  (\cite[Theorem 3.2]{KPU11}), and from here  lifted  to the whole ring $\K[\T,\X,\Y]$.

To compare this to our approach, note that the back face of the polytope in Figure~\ref{PolytopeOfA} is the polygon $P$ appearing in Figure~\ref{polygonfig}, which gives the rational normal surface $S_{\mu_1,\mu_2} \subseteq \PP^{\mu+1}$.  For us, $S_{\mu_1,\mu_2}$ is more natural geometrically (see \eqref{BGIdiagram}), while \cite{KPU11} uses $S_{1,\mu_1,\mu_2}$ because it leads naturally to the surjection \eqref{AtoRees}.

In a different direction, our situation was also studied  by Lin and Polini in \cite{LP14}.  Given a homogeneous complete intersection $J = \langle f_1,\dots,f_r\rangle$ in a polynomial ring and an integer $d\ge \max\{\deg(f_i)\}$, they consider the truncation $J_{\ge d}$ of $J$ in degrees $\ge d$ and prove several results about the Rees algebra $\cR(J_{\ge d})$, including a complete description when $d$ is large (\cite[Theorem 3.9]{LP14}).  To relate their results to our situation, note that the ideal $J=\langle\alpha_{d-\mu_1},\,\beta_{d-\mu_2}\rangle\subseteq\K[\T]$ is clearly a complete intersection, and it is straightforward to check that our ideal $I = \langle \alpha_{d-\mu_1} \T^{\mu_1}, \beta_{d-\mu_2} \T^{\mu_2}\rangle$ is precisely $J_{\ge d}$.  To analyze this case, Lin and Polini set ${\mathfrak M}=\langle T_0, T_1\rangle \subseteq \K[\T]$ and $M = {\mathfrak M}^{\mu_1}(\mu_1-d)\oplus{\mathfrak M}^{\mu_2}(\mu_2-d)$.  The surjection $M \twoheadrightarrow I$ defined by $(u,v)\mapsto u\alpha_{d-\mu_1}+v\beta_{d-\mu_2}$ gives a surjection $\cR(M)\twoheadrightarrow \cR(I)$ of Rees algebras, which turns out to be precisely the map \eqref{AtoRees}.  
Their main result about $\cR(I)$ in this case (\cite[Theorem 3.11]{LP14}) is based on the methods of \cite{KPU11}.

Subsequently, in \cite[Example 3.20]{Madsen}, Madsen studied the Rees algebra of an arbitrary height two ideal of $\K[\T]$ generated in degree $d$.  Madsen builds up the Hilbert-Burch matrix of the given ideal one column at a time.  The cokernels $E_1, E_2, \dots$ of the partial Hilbert-Burch matrices give surjections 
\[
E_1 \twoheadrightarrow E_2 \twoheadrightarrow  \cdots
\]
that induce surjections of Rees algebras
\begin{equation}
\label{madsenapproach}
S \twoheadrightarrow \cR(E_1) \twoheadrightarrow \cR(E_2) \twoheadrightarrow  \cdots
\end{equation}
for a suitable polynomal ring $S$.  If $\cR(E_i) = S/\cK_{i}$, then $\cK_i \subseteq \cK_{i+1}$ and 
\[
\cR(E_{i+1}) = \cR(E_i)/\cK_{i+1}\cR(E_i).
\]
By the incremental construction of the $E_i$'s, $\cK_{i+1}\cR(E_i)$ has a relatively simple description (\cite[Proposition 3.1]{Madsen}).  The main result (\cite[Theorem 3.9]{Madsen}) describes some of the minimal generators of $\cK_{i+1}\cR(E_i)$.  

When applied to our ideal $I= \langle \alpha_{d-\mu_1} \T^{\mu_1}, \beta_{d-\mu_2} \T^{\mu_2}\rangle$, Madsen's results are strong enough to give a complete description of the defining equations $\cI$ of $\cR(I)$.  In the sequence of Rees algebras \eqref{madsenapproach}, the last two turn out to be exactly \eqref{AtoRees}, which in the notation of \cite[Example 3.20]{Madsen} is written $\cR(E) \to \cR(I)$.  Madsen uses the modules $M = {\mathfrak M}^{\mu_1}(\mu_1-d)\oplus{\mathfrak M}^{\mu_2}(\mu_2-d)$ and $F = E^{**}$ and notes that $E \simeq M$ and $F \simeq \K[\T](\mu_1-d)\oplus \K[\T](\mu_2-d)$, so that $E = F_{\ge d}$.  (Madsen works with $I(d)$ rather than $I$, so $M$ is written ${\mathfrak M}^{\mu_1}(\mu_1)\oplus{\mathfrak M}^{\mu_2}(\mu_2)$ in \cite{Madsen}.)

Madsen also notes that $\cR(F)$ is the ring $\K[\T,X,Y]$ defined in Section~\ref{1}.  The inclusion $E \subseteq F$ induces a map of Rees algebras $\cR(E) \hookrightarrow \cR(F) =  \K[\T,X,Y]$ that fits into the commutative diagram
\begin{equation}
\label{Madsen.sec3}
\begin{array}{c}
\SelectTips{cm}{}
\xymatrix@C=10pt{
\K[\T,\X,\Y] 
\ar@{->>}[d]^(.55){\Phi'}   \ar@{->>}[ddr]^\Phi\\
\cR(E)  \ar@{^{(}->}[d] \ar[dr]\\
 \K[\T,X,Y] \ar[dr]^\phi & \cR(I)\ar@{^{(}->}[d] \\
 & \K[\T,s]}
\end{array}
\end{equation}
where $\Phi', \Phi, \phi$ are from the diagram \eqref{Reesdiagram}.  Madsen's approach refines \eqref{Reesdiagram} by regarding $\Phi'$ and $\Phi$ as surjections onto their images, which are the Rees algebras $\cR(E)$ and $\cR(I)$  shown in \eqref{Madsen.sec3}. 

Since $\cI' = \ker(\Phi')$ and $\cR(E) = \K[\T,\X,\Y]/\cI'$, Madsen observes that it suffices to know $\cI \cR(E)$, which is given by the intersection
\begin{equation}
\label{sec3.intersection1}
\cI \cR(E) = ( g\cR(F)) \cap \cR(E),
\end{equation}
where $g = \alpha_{d-\mu_1}Y - \beta_{d-\mu_2}X$ as in the proof of Theorem~\ref{mtm}.  
Note that the height one prime ideal $\cI A$ in \cite{KPU11} is precisely $\cI \cR(E)$ is  since the ring $A$ in \eqref{AtoRees} is $\cR(E)$.

In \cite[Example 3.20]{Madsen}, Madsen refines \eqref{sec3.intersection1} to the stronger equality
\begin{equation}
\label{sec3.intersection2}
\cI \cR(E) = (g\cR(F))_{\ge0,*}.
\end{equation}
Madsen uses earlier results in the paper to explain which multiples occur and how they lift back to $\K[\T,\X,\Y]$.  Combining these new generators with $\cI'$, Madsen obtains a complete description of the defining equations $\cI$ of $\cR(I)$.  

For us, \eqref{sec3.intersection2} is \eqref{2.9.claim} since $\Phi'(\cI) = \cI\cR(E)$. Our proof uses the explicit construction of the $\Psi^t_{i_\ell,j_\ell,k_\ell+1} \in \cI$, while for Madsen, \eqref{sec3.intersection2}  is a special case of a more general result (see \cite[(3.11)]{Madsen}).  

Overall, our treatment of $\cI$ is consistent with what Madsen does;\ the main difference is that we use elementary methods that avoid the machinery of \cite{Madsen}.  

\section{The Rees Algebra of the Plane Curve}\label{4}
Consider now a rational parametrization $\PP^1\to\PP^2$ of a genus zero algebraic curve $C\subseteq\PP^2$ of degree $d$ defined by homogeneous elements $f_{0,d},f_{1,d},f_{2,d} \in\K[\T]$ of degree $d$ and $\gcd(f_{0,d},f_{1,d},f_{2,d})=1:$
\begin{equation}\label{f}
\begin{array}{cccc}
{\bf f}:&\PP^1&\to&\PP^2\\
&\bft&\mapsto & (f_{0,d}(\bft): f_{1,d}(\bft): f_{2,d}(\bft))
\end{array}
\end{equation}
We assume in addition that $d > 1$ and $f_{0,d},f_{1,d},f_{2,d}$ are linearly independent.
In this section we consider the Rees algebra $\cR(K)$ of the ideal $K = \langle f_{0,d},f_{1,d},f_{2,d}\rangle \subseteq \K[\T]$ and explain how $\bff$ generates the diagrams \eqref{BGIdiagram} and \eqref{Reesdiagram} from the Introduction.

Let the $\mu$-basis of $f_{0,d},f_{1,d},f_{2,d}$ be
\[
p = \big(p_{0,\mu},p_{1,\mu},p_{2,\mu}\big)\in\K[\T]_{\mu}^3,\, q= \big(q_{0,d-\mu},q_{1,d-\mu},q_{2,d-\mu}\big)\in\K[\T]_{d-\mu}^3,
\]
where as usual $\mu \le d-\mu$.  Following \cite{BGI}, we let
\[
A =\big(A_{0,\mu_1},A_{1,\mu_1},A_{2,\mu_1}\big)\in\K[\T]_{\mu_1}^3,\,B=\big(B_{0,\mu_2},B_{1,\mu_2},B_{2,\mu_2}\big)\in\K[\T]_{\mu_2}^3
\]
be a $\mu$-basis of $(p_{0,\mu},p_{1,\mu},p_{2,\mu})$ with $0 \le \mu_1 \le \mu_2$ and $\mu = \mu_1 + \mu_2$. If  $\mu_1<\mu_2,$
then $A$ is uniquely defined up to a constant, and $B$ uniquely defined (up to a constant also) modulo $A$. If $\mu_1=\mu_2$, then there are infinitely possible choices of $A$ up to a constant.

\begin{remark}
\label{nonunique}
If $\mu< d-\mu$, then $p$ is unique up to a constant, which implies that $\mu_1$ is uniquely determined.  However, if $\mu=d-\mu$, then different choices of $p$ can lead to different values of $\mu_1$. For example, suppose that $d=6$ and $(f_{0,6},f_{1,6},f_{2,6})$ parametrizes a rational sextic in $\PP^2$ with $\mu=3$.  If the curve has three triple points, then \cite[Lem.\ 4.14]{CKPU} implies that in suitable coordinates, the Hilbert-Burch matrix becomes
\[
\begin{pmatrix} Q_1 & Q_1\\ Q_2 & 0\\ 0 & Q_3\end{pmatrix},
\]
where $Q_1, Q_2, Q_3$ are linearly independent cubics.  We can choose $p$ to be any nonzero vector in the column space.  Writing $p$ as a row, we have:
\begin{align*}
p &= (Q_1,Q_2,0) \text{ has $\mu_1 = 0$ since $Q_1,Q_2,0$ are linearly dependent.}\\
p &= (2Q_1,Q_2,Q_3) \text{has $\mu_1 = 1$ since $2Q_1,Q_2,Q_3$ are linearly independent.}
\end{align*}
As we vary over all $p$, the generic value is $\mu_1=1$, but up to a constant, there are three choices of $p$ with $\mu_1 = 0$.
\end{remark}

Since $(p_{0,\mu},p_{1,\mu},p_{2,\mu})$ is a syzygy on $(f_{0,d},f_{1,d},f_{2,d})$, the latter is a syzygy on the former, so that
$(f_{0,d},f_{1,d},f_{2,d})$ can be decomposed as
\begin{equation}
\label{decom}
\big(f_{0,d},f_{1,d},f_{2,d}\big) = \alpha_{d-\mu_1}\big(A_{0,\mu_1},A_{1,\mu_1},A_{2,\mu_1}\big)+\beta_{d,-\mu_2}\big(B_{0,\mu_2},B_{1,\mu_2},B_{2,\mu_2}\big),
\end{equation}
with $\alpha_{d-\mu_1},\,\beta_{d-\mu_2}\in\K[\T]$ homogeneous of degrees $d-\mu_1$ and $d-\mu_2$ respectively.  Since $\alpha_{d-\mu_1},\ \beta_{d-\mu_2}$ clearly have no common factors, they give a parametrization $\gamma : \PP^1 \to \PP^{\mu+1}$ as in \eqref{gamma} with image contained in $\cS_{\mu_1,\mu_2}$.

To relate these parametrized curves geometrically, we write for $i=0,1,2$:
\begin{equation}
\label{AB}
\begin{array}{c}
\begin{aligned}
A_{i,\mu_1}&=\sum_{j=0}^{\mu_1} a_{i,j}T_0^{\mu_1-j}T_1^j = \ua_i\cdot \T^{\mu_1}\\
B_{i,\mu_2}&=\sum_{j=0}^{\mu_2}b_{i,j}T_0^{\mu_2-j}T_1^j  = \ub_i\cdot \T^{\mu_2}.
\end{aligned}
\end{array}
\end{equation}
Then we get the projection $\PP^{\mu+1} \dashrightarrow \PP^2$ defined by
\begin{equation}
\label{projection}
(\X:\Y) \longmapsto \uZ = (\ua_0\!\cdot \X + \ub_0\!\cdot \Y:\ua_1\!\cdot \X + \ub_1\!\cdot \Y:\ua_2\!\cdot \X + \ub_2\!\cdot \Y),
\end{equation}
where $\uZ=(Z_0,Z_1,Z_2)$ are homogeneous coordinates for $\PP^2$.

The projection \eqref{projection} interacts nicely with the scroll $\cS_{\mu_1,\mu_2} \subseteq \PP^{\mu+1}$.  First, the computations
\begin{align*}
(\t^{\mu_1}:0) &\longmapsto (\ua_0\!\cdot \t^{\mu_1}:\ua_1\!\cdot \t^{\mu_1}:\ua_2\!\cdot \t^{\mu_1}) = (A_{0,\mu_1}(\t):A_{1,\mu_1}(\t): A_{2,\mu_1}(\t)) = A(\t)\\
(0:\t^{\mu_2}) &\longmapsto (\ub_0\!\cdot \t^{\mu_2}:\hskip1pt\ub_1\!\cdot \t^{\mu_2}:\hskip1pt\ub_2\!\cdot \t^{\mu_2}) = (B_{0,\mu_2}(\t):\hskip.5ptB_{1,\mu_2}(\t): \hskip.5ptB_{2,\mu_2}(\t)) = B(\t)
\end{align*}
show that the rational normal curves that form the edges of the scroll project to the curves given by the syzygies $A,B$. Furthermore, by Hilbert-Burch, $p $ is equal up to a nonzero constant to the signed maximal minors of the $2\times3$ matrix made by $A$ and  $B.$ Note that $p$ has no basepoints since it is a minimal syzygy.  Hence $A(\t),B(\t)$ are always distinct points in $\PP^2$.

 It follows that the line of the scroll for parameter $\t$ projects to the line through $A(\t),B(\t)$, which is the moving line defined by $p$. We also have that
by \eqref{decom}, $\gamma(\t) = (\alpha_{d-\mu_1}(\t)\hskip1pt\t^{\mu_1}:\beta_{d-\mu_2}(\t)\hskip1pt\t^{\mu_2})$ projects to $(f_{0,d}(\t):f_{1,d}(\t):f_{2,d}(\t))$.

All these show that the projection $\PP^{\mu+1} \dashrightarrow \PP^2$ induces a morphism $\cS_{\mu_1,\mu_2} \rightarrow \PP^2$.  It follows that we get the commutative diagram from \eqref{BGIdiagram}:
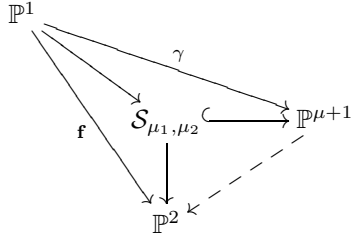
\[
\SelectTips{cm}{}
\xymatrix@C=30pt{
\PP^1 \ar[drr]^\gamma \ar[dr] \ar[ddr]_\bff& & \\
& \cS_{\mu_1,\mu_2} \ar@{^{(}->}[r] \ar[d] & \PP^{\mu+1} \ar@{-->}[dl] \\
& \PP^2 & }
\]

From the algebraic point of view, the projection \eqref{projection} gives the map:
\begin{equation}
\label{Gama}
\begin{array}{rcl}
\Gamma:\K[\T, \uZ]&\longrightarrow&\K[\T,\X,\Y] \\
T_i&\longmapsto& T_i,\ i=0,1\\
Z_j&\longmapsto& \ua_j \cdot \X + \ub_j \cdot \Y,\ j=0,1,2.
\end{array}
\end{equation}
For the Rees algebra of $K=\langle f_{0,d},f_{1,d},f_{2,d}\rangle \subseteq \K[\T]$, we have the map:
\begin{equation}
\label{psidef}
\begin{array}{rcl}
\psi:\K[\T,\uZ]&\longrightarrow&\K[\T,s]\\
T_i&\longmapsto & T_i, \ i=0,1\\
Z_j&\longmapsto &f_{j,d}\,s, \ j=0, 1, 2,
\end{array}
\end{equation}
whose image is $\cR(K)$.  The kernel $\cK=\ker{\psi} \subseteq \K[\T,\uZ]$ is the moving curve ideal of the parametrized curve in $\PP^2$ and gives the equations defining the Rees algebra.

To see how $\Gamma$ and $\psi$ relate to the map $\phi$ from Section~\ref{1} and maps $\Phi$, $\Phi'$ from Section \ref{2}, we need to introduce one more map:
\begin{equation}
\label{Omega}
\begin{array}{rcl}
\Omega:\K[\T,\uZ]&\longrightarrow&\K[\T,X,Y]\\
T_i&\longmapsto & T_i,\, i=0,1\\
Z_j&\longmapsto & A_{j,\mu_1}X+B_{j,\mu_2}Y, \,j=0,1,2.
\end{array}
\end{equation}
If we use the bigrading on $\K[\T,\uZ]$ defined by $\deg(T_i) = (1,0)$ and $\deg(Z_i) = (0,1)$, then one can check that $\Gamma$, $\psi$ and $\Omega$ all preserve the bigradings.

\begin{lemma}
\label{maplemma}
The maps $\Phi$, $\Phi'$, $\Gamma$, $\Omega$, $\phi$, and $\psi$ defined above fit together into the commutative diagram from \eqref{Reesdiagram}{\rm:}
\[
\xymatrix{
& \K[\T,\X,\Y] \ar[d]^{\Phi'} \ar[ddr]^\Phi & \\
\K[\T,\uZ] \ar[ur]^\Gamma \ar[r]^\Omega \ar[rrd]_\psi & \K[\T,X,Y] \ar[dr]^\phi & \\
&& \K[\T,s]}
\]
\end{lemma}

\begin{proof} We have already observed that $\Phi = \phi\circ \Phi'$.  Then notice that
\[
Z_j \stackrel{\Gamma}{\longmapsto} \sum_{k=0}^{\mu_1}a_{j,k}X_{k}+ \sum_{k=0}^{\mu_2}b_{j,k}Y_{k} \stackrel{\Phi'}{\longmapsto}\sum_{k=0}^{\mu_1}a_{j,k}(T_0^{\mu_1-k}T_1^k X)+ \sum_{k=0}^{\mu_2}b_{j,k}(T_0^{\mu_2-k}T_1^k Y).
\]
The expression on the right equals $A_{j,\mu_1}X+B_{j,\mu_2}Y = \Omega(Z_j)$, and $\Omega = \Phi' \circ \Gamma$  follows.  Finally, we have
\[
Z_j \stackrel{\Omega}{\longmapsto} A_{j,\mu_1}X+B_{j,\mu_2}Y \stackrel{\phi}{\longmapsto} A_{j,\mu_1}\alpha s +B_{j,\mu_2}\beta_{d-\mu_2} s
\]
The expression on the right equals $f_{jd,}s = \psi(Z_j)$, and $\psi =\phi \circ \Omega$  follows.
\end{proof}

\begin{corollary}
\label{mapcor}
The ideal 
$\cK$ is equal to the inverse image via $\Omega$ of $\langle \alpha_{d-\mu_1} Y - \beta_{d-\mu_2} X\rangle$.
\end{corollary}

\begin{proof}
Lemma~\ref{maplemma} implies that $\cK = \ker(\psi) = \ker(\phi \circ \Omega) = \Omega^{-1}(\ker(\phi))$.  Then we are done since $\ker(\phi) = \langle \alpha_{d-\mu_1} Y - \beta_{d-\mu_2} X\rangle$ by Lemma~\ref{alphabetaRees}.
\end{proof}

As is standard, the syzygy $(p_{0,\mu},p_{1,\mu},p_{2,\mu})$ gives the polynomial
\begin{equation}\label{p}
p:=p_{0,\mu}Z_0+p_{1,\mu}Z_1+p_{2,\mu}Z_2 =
\det\!\left(\!\!\begin{array}{ccc}
Z_0&Z_1&Z_2\\
A_{0,\mu_1}&A_{1,\mu_1}&A_{2,\mu_1}\\
B_{0,\mu_2}&B_{1,\mu_2}&B_{2,\mu_2}
\end{array}\!\!
\right)\in\K[\T,\uZ].
\end{equation}
Note that $p$ is an element of bidegree $(\mu,1)$ which vanishes after specializing $Z_j\mapsto A_{j,\mu_1}X+B_{j,\mu_2}Y, \,j=0,1,2$, and hence belongs to $\cK$. Moreover, $p\in\ker(\Omega)$. Actually, it generates the whole kernel.

\begin{proposition}
\label{kerOmega}
The ideal generated by  $p$ is equal to $\ker(\Omega).$
\end{proposition}

\begin{proof}
We can assume w.l.o.g.\ that $F_{i,j}\in\ker(\Omega)$ is primitive with respect to the $\T$-variables.
From \eqref{p}, we get that if we specialize $(Z_0,Z_1,Z_2)$ in $\overline{\K(\T)}^3$, then $p=0$ if and only if there exist $\lambda,\nu\in\overline{\K(\T)}$ such that
$$\uZ=\lambda \big(A_{0,\mu_1},A_{1,\mu_1},A_{2,\mu_2}\big)+\nu\big(B_{0,\mu_2},B_{1,\mu_2},B_{2,\mu_2}\big).$$
For each of these $\lambda,\mu$, we set $x=\lambda, \,y=\mu$, and get that $$F_{i,j}\big(\T,A_{0,\mu_1}x+B_{0,\mu_2}y,A_{1,\mu_1}x+B_{1,\mu_2}y,A_{2,\mu_1}x+B_{2,\mu_2}y \big)=0.$$ By the Nullstellensatz, we have that $p$ divides $F_{i,j}$ in $\K(\T)[\uZ]$. As both $p$ and $F_{i,j}$ are primitive with respect to the $\T$-variables, the division actually holds in $\K[\T,\uZ]$.
\end{proof}

The following result shows that in some bidegrees, all we need is $p$.

\begin{corollary}\label{belowbottom}
If $F_{i,j}\in \cK$ and $i+\mu_2j<d-\mu_1$, then it is a multiple of $p$.
\end{corollary}

\begin{proof}
If $\Omega(F_{i,j})$ is not zero, it should be a polynomial of $\T$-degree at least $d-\mu_1$. On the other hand, this polynomial has $\T$-degree $i+\mu_2j$.  From here, the claim follows straightforwardly.
\end{proof}

\begin{remark}\label{rem}\
\begin{enumerate}
\item In Figure \ref{triangularregion} at the end of Section~\ref{5}, Corollary~\ref{belowbottom} shows that in bidegrees that lie strictly below the bottom edge of the triangular region in the figure, $\cK$ is generated by $p$.
\item Corollary~\ref{belowbottom} is a slight strengthening of Theorem 2.10(3) of \cite{CD2}.
\end{enumerate}
\end{remark}

\section{The Other Syzygy and Some Explicit Minimal Generators}\label{5}

So far, the syzygy $p$ of degree $\mu$ has played a
central role.  But what about the other syzygy $q = q_{0,d-\mu}Z_0+q_{1,d-\mu}Z_1+q_{2,d-\mu}Z_2$ of
degree $d-\mu$?   Our next result shows that it maps via $\Omega$ to $-(\alpha_{d-\mu_1} Y-\beta_{d-\mu_2} X)$.

\begin{proposition}\label{q}
\label{abprop}
With notation as above, we have that
$$\Omega(q)=-(\alpha_{d-\mu_1} Y - \beta_{d-\mu_2} X).$$
\end{proposition}

\begin{proof}
To prove the claim, by using \eqref{Omega}, we have to show that
\begin{equation}\label{xio}
\begin{aligned}
\alpha_{d-\mu_1} &= -(q_{0,d-\mu} B_{0,\mu_2} + q_{1,d-\mu} B_{1,\mu_2} + q_{2,d-\mu} B_{2,\mu_2})\\
\beta_{d-\mu_2}  &= q_{0,d-\mu} A_{0,\mu_1} + q_{1,d-\mu} A_{1,\mu_1} + q_{2,d-\mu} A_{2,\mu_1}.
\end{aligned}
\end{equation}
Since $(f_{0,d}, f_{1,d}, f_{2,d})$ is given by the $2\times2$ minors (with signs)
of its Hilbert-Burch matrix, we have
\begin{align*}
f_{0,d} &= p_{1,\mu} q_{2,d-\mu} - p_{2,\mu} q_{1,d-\mu},\ f_{1,d}=  p_{2,\mu} q_{0,d-\mu} - p_{0,\mu} q_{2,d-\mu},\ f_{2,d}\\ &= p_{0,\mu} q_{1,d-\mu} -
  p_{1,\mu} q_{0,d-\mu}\\
p_{0,\mu} &= A_{1,\mu_1} B_{2,\mu_2} - A_{2,\mu_1} B_{1,\mu_2},\ p_{1,\mu} = A_{2,\mu_1} B_{0,\mu_2} - A_{0,\mu_1} B_{2,\mu_2},\ p_{2,\mu}\\ &= A_{0,\mu_1} B_{1,\mu_2} -
  A_{1,\mu_1} B_{0,\mu_2}.
\end{align*}
We deduce then that
\begin{align*}
f_{0,d} &= p_{1,\mu} q_{2,d-\mu} - p_{2,\mu} q_{1,d-\mu}\\ &= (A_{2,\mu_1} B_{0,\mu_2} - A_{0,\mu_1} B_{2,\mu_2})q_{2,d-\mu} - (A_{0,\mu_1} B_{1,\mu_2} -
  A_{1,\mu_1} B_{0,\mu_2})q_{1,d-\mu}\\
&= (-B_{1,\mu_2} q_{1,d-\mu} - B_{2,\mu_2} q_{2,d-\mu})A_{0,\mu_1} + (A_{1,\mu_1} q_{1,d-\mu} + A_{2,\mu_1} q_{2,d-\mu})B_{0,\mu_2}\\
&= (-q_{1,d-\mu} B_{1,\mu_2} - q_{2,d-\mu} B_{2,\mu_2})A_{0,\mu_1} + (q_{1,d-\mu} A_{1,\mu_1} + q_{2,d-\mu} A_{2,\mu_1})B_{0,\mu_2}\\
&= \alpha_{d-\mu_1}'A_{0,\mu_1} + \beta_{d-\mu_2}'B_{0,\mu_2},
\end{align*}
with $$
\begin{array}{ccl}
\alpha_{d-\mu_1}'& :=& -(q_{0,d-\mu} B_{0,\mu_2} + q_{1,d-\mu} B_{1,\mu_2} + q_{2,d-\mu} B_{2,\mu_2}),\\
\beta_{d-\mu_2}'&:=&q_{0,d-\mu} A_{0,\mu_1} + q_{1,d-\mu}
A_{1,\mu_1} + q_{2,d-\mu} A_{2,\mu_1}.
\end{array}$$  Similarly,
we get
\begin{align*}
f_{1d,} = p_{2,\mu} q_{0,d-\mu} - p_{0,\mu} q_{2,d-\mu}
&= \alpha_{d-\mu_1}'A_{1,\mu_1} + \beta_{d-\mu_2}'B_{1,\mu_2},\\
f_{2,d} = p_{0,\mu} q_{1,d-\mu} - p_{1,\mu} q_{0,d-\mu} & = \alpha_{d-\mu_1}'A_{2,\mu_1} + \beta_{d-\mu_2}'B_{2,\mu_2}.
\end{align*}
This shows that $\alpha_{d-\mu_1}'A + \beta_{d-\mu_2}'B = (f_{0,d},f_{1,d},f_{2,d})$, and $\alpha_{d-\mu_1} =
\alpha_{d-\mu_1}'$, $\beta_{d-\mu_2} = \beta_{d-\mu_2}'$ follows since $A,B$ are a basis of the
syzygy module of $(p_{0,\mu}, p_{1,\mu}, p_{2,\mu})$.
\end{proof}

To produce more elements which are mapped to a multiple of $\alpha_{d-\mu_1} Y-\beta_{d-\mu_2} X$ via $\Omega$, we use the following regularity result.

\begin{proposition}\label{cc}
If $i\geq\mu+\mu_2-1$, then we have $\langle p_{0,\mu}, p_{1,\mu}, p_{2,\mu} \rangle_{i,j} =  \K[\T,\uZ]_{i,j}$.
\end{proposition}

\begin{proof}
Let $I = \langle p_{0,\mu}, p_{1,\mu}, p_{2,\mu} \rangle \subseteq R = \K[\T]$.  It suffices to prove that $I_i = \K[T]_i$ for $i \geq\mu+\mu_2-1$. We have the exact sequence
\[
0 \longrightarrow R_{i-\mu-\mu_1}\oplus R_{i-\mu-\mu_2} \longrightarrow R_{i-\mu}^3 \xrightarrow{(p_{0,\mu},p_{1,\mu},p_{2,\mu})} I_i \longrightarrow 0.
\]
Note that $i-\mu \ge i-\mu-\mu_1 \ge i-\mu-\mu_2$.  In general, $\dim R_m = m+1$ for all $m \ge -1$.  Thus, if $i-\mu-\mu_2 \ge -1$, then the above exact sequence implies
\begin{align*}
\dim I_i &= 3(i-\mu+1) - (i-\mu-\mu_1+1) - (i-\mu-\mu_1 +1) \\
&=i - \mu + \mu_1+\mu_2 + 1 = i+1.
\end{align*}
Since $\dim R_i = i+1$, it follows that $I_i = R_i$ when $i-\mu-\mu_2 \ge -1$, i.e., when $i \ge \mu+\mu_2-1$.
\end{proof}

With this result in mind, we proceed as follows:\ let $F_{i,j}\in\langle p_{0,\mu},p_{1,\mu}, p_{2,\mu}\rangle$ (this always holds for instance if $i\geq\mu+\mu_2-1$ thanks to Proposition \ref{cc}), and write
$$F_{i,j}=\sum_{\ell=0}^2p_{\ell,\mu}\,F^{(\ell)}_{i-\mu,j},
$$
for suitable homogeneous elements $F^{(\ell)}_{i-\mu,j}\in\K[\T,\uZ],\,\ell=0,1,2$. Then set
\begin{equation}\label{DB}
D_B\big(F_{i,j}\big):=\det\left(\begin{array}{ccc}
F^{(0)}_{i-\mu,j}&F^{(1)}_{i-\mu,j}&F^{(2)}_{i-\mu,j}\\
Z_0&Z_1&Z_2\\
A_{0,\mu_1}& A_{1,\mu_1}& A_{2,\mu_1}
\end{array}
\right).
\end{equation}
Note that $D_B(F_{i,j})$ has bidegree $(i-\mu_2,j+1)$.
Similarly, $D_A(F_{i,j})$ of bidegree $(i-\mu_1,j+1)$ is defined by replacing the last row of the matrix in \eqref{DB} with $B_{0,\mu_2} B_{1,\mu_2} B_{2,\mu_2}$. If the image of these operators lies in $\langle p_{0,\mu},\,p_{1,\mu},\,p_{2,\mu}\rangle$, one can iterate them again to get  $D_A^aD_B^b(F_{i,j})$. The following result is straightforward.

\begin{proposition}\label{DaDb}
Let $F_{i,j}\in\K[\T,\uZ]$.  If $D_A^aD_B^b(F_{i,j})$ is defined, then it is an element of $\K[\T,\uZ]$ of bidegree $(i-a\mu_1-b\mu_2, j+a+b)$ such that
\begin{equation}\label{formula}
\Omega\left(D_A^aD_B^b(F_{i,j})\right) = (-1)^b X^aY^b \Omega\left(F_{i,j}\right).
\end{equation}
Furthermore, $F_{i,j}$ belongs to $\cK$ if and only if $D_A^aD_B^b(F_{i,j})$ belongs to $\cK$.
\end{proposition}

\begin{proof}
Since $\Omega$ is the identity on $\K[\T]$, applying $\Omega$ to $D_B(F_{i,j})$ gives
\begin{align*}
\Omega(D_B(F_{i,j})) &= \det\begin{pmatrix}
\Omega(F^{(0)}_{i-\mu,j})&\Omega(F^{(1)}_{i-\mu,j}) &\Omega(F^{(2)}_{i-\mu,j})\\[2pt] X A_{0,\mu_1} + Y B_{0,\mu_2} & X A_{1,\mu_2} + Y B_{1,\mu_2} & X A_{2,\mu_2} + Y B_{2,\mu_2} \\ A_{0,\mu_1} & A_{1,\mu_1} & A_{2,\mu_1}\end{pmatrix} \\ &=
  \det\begin{pmatrix} \Omega(F^{(0)}_{i-\mu,j})&\Omega(F^{(1)}_{i-\mu,j}) &\Omega(F^{(2)}_{i-\mu,j})\\[2pt] Y B_{0,\mu_2} & Y B_{1,\mu_2} & Y B_{2,\mu_2} \\ A_{0,\mu_1} & A_{1,\mu_1} & A_{2,\mu_1}\end{pmatrix}  \\
  &= Y  \det\begin{pmatrix}  \Omega(F^{(0)}_{i-\mu,j})&\Omega(F^{(1)}_{i-\mu,j}) &\Omega(F^{(2)}_{i-\mu,j}) \\[2pt] B_{0,\mu_2} & B_{1,\mu_2} & B_{2,\mu_2} \\ A_{0,\mu_1} & A_{1,\mu_1} & A_{2,\mu_1}\end{pmatrix}\\ &= -Y\left(\sum_{\ell=0}^2p_{\ell,\mu}\,\Omega(F^{(\ell)}_{i-\mu,j})\right)= -Y\Omega(F_{i,j}),
  \end{align*}
where the last line holds by \eqref{p}.  By recurrence, \eqref{formula} follows.

To prove the last part of the claim, thanks to Corollary \ref{mapcor} we have that $F_{ij}\in\cK$ if and only if $\Omega(F_{ij})=(\alpha_{d-\mu_1}Y-\beta_{d-\mu_2}X)G$
for a suitable $G\in\K[\T,X,Y].$ This combined with \eqref{formula} prove the claim.
\end{proof}

From Propositions \ref{q} and \ref{DaDb}, we deduce straightforwardly the following result.

\begin{corollary}\label{pi}
If $D_A^aD_B^b(q)$ is defined, then
\begin{equation}\label{ki}
\Omega\big(D_A^aD_B^b(q)\big)=(-1)^{b+1} X^aY^b\big(\alpha_{d-\mu_1} Y -\beta_{d-\mu_2} X \big).
\end{equation}
\end{corollary}

Thanks to Proposition \ref{cc}, we have the following:

\begin{lemma}\label{ax}
With notation as above, $D_A^aD_B^b(F_{i,j})$ is defined whenever $a \ge 0$ and either
$b \ge 1, i-a\mu_1-b\mu_2\geq\mu-1$
or  $b = 0, i-a\mu_1\geq\mu+\mu_2-\mu_1-1$.
\end{lemma}

Proposition \ref{indep} below shows that \eqref{ki} actually produce  some nice minimal generators of $\cK$.

\begin{proposition}\label{indep}
If $\mu_1>0$ and the family $\{F_{i_1,j_1},\dotsb, F_{i_\ell,j_\ell}\}\subset\cK$ is such that $\Omega\big(F_{i_k,j_k}\big)=X^{a_k}Y^{b_k}(\alpha_{d-\mu_1}Y-\beta_{d-\mu_2}X)$ for $k=1,\dotsb, \ell$, where $(a_{k'},b_{k'})\neq(a_k,b_k)$ if $k\neq k'$, then this family is contained in a system of minimal generators of $\cK$.
\end{proposition}
\begin{proof}
Let $\{G_1,\dotsb, G_m\}$ be a family of minimal generators of $\cK$. For each $k=1,\dotsb, \ell$, as $F_{{i_k,j_k}}\in\cK$, we must have
\begin{equation}\label{ixx}
F_{{i_k,j_k}}=\sum_{\ell=1}^m R_\ell G_\ell
\end{equation}
for suitable bihomogeneous polynomials $R_1,\dotsb, R_m\in\K[\T,\uZ]$.
By applying $\Omega$ to both sides of this expression, we get that (thanks to  Corollary \ref{mapcor} and the hypothesis)
\begin{equation}\label{abbov}
X^{a_k}Y^{b_k}(\alpha_{d-\mu_1}Y-\beta_{d-\mu_2}X)=(\alpha_{d-\mu_1}Y-\beta_{d-\mu_2}X)\bigg(\sum_{\ell=1}^m \Omega(R_\ell)G^*_\ell\bigg),
\end{equation}
with $G^*_\ell\in\K[\T,X,Y]$. Canceling the common factor in both sides of \eqref{abbov}, we obtain
\begin{equation}\label{qui}
X^{a_k}Y^{b_k}=\bigg(\sum_{\ell=1}^m \Omega(R_\ell)G^*_\ell\bigg).
\end{equation}
From the definition of $\Omega$ given in \eqref{Omega}, and using the fact that $\mu_1, \mu_2>0,$  we deduce that $\Omega(R_\ell)\in\langle T_0,\,T_1\rangle$ unless $\deg_\T(R_\ell)=\deg_\uZ(R_\ell)=0$. So, there must be an index $\ell_0\in\{1,\dotsb, m\}$ such that  $R_{\ell_0}=\lambda_{s_0}\in\K^\times$, and $G^*_{\ell_0}$ has $X^{a_k}Y^{b_k}$ among its monomials. Hence,
from \eqref{ixx} we get
$$
F_{{i_k,j_k}}-\lambda_{\ell_0}G_{\ell_0}=\sum_{\ell\neq \ell_0} R_\ell G_\ell,
$$
which implies  straightforwardly that both families $\{G_1,\dotsb, G_m\}$ and \newline $\{G_1,\dotsb, G_{\ell_0-1}, F_{i_k,j_k}, G_{\ell_0+1}, G_m\}$
are minimal generators of $\cK$.

To conclude, we have to show that we can add {\em all} the   $F_{{i_{k'},j_{k'}}}$ with $k'\neq k$ to the list of minimal generators.  This can be done recursively following the reasoning given above, just noting that in each step of  the process the $R_s$ which is mapped via $\Omega$ to a constant $\lambda_s\in\K^\times$ can always be chosen among those remaining $G_s$ in the list, which is straightforward. This concludes with the proof of the proposition.
\end{proof}
\begin{remark}\label{rrem}
The hypothesis $\mu_1>0$ is necessary in Proposition \ref{indep}. Indeed, if $\mu_1=0$, we may have a $\K$-linear combination of the $Z_i$'s that get mapped to a nonzero scalar multiple of $X$ in $\K[\T,X,Y]$ (for instance if the coordinates of $B$ are $\K$-linearly independent).
\end{remark}
\begin{theorem}\label{DD}
Assume that $\mu_1>0.$ Then a subset of minimal generators of $\cK$ is given by the polynomials
$D_A^aD_B^b(q)$ with $a \ge 0$ and either
$b \ge 1, d-\mu-a\mu_1-b\mu_2\geq\mu-1$
or  $b = 0, d-\mu-a\mu_1\geq\mu+\mu_2-\mu_1-1$.
\end{theorem}
\begin{proof}
Consider the region in the first quadrant whose lattice points are given by
\[
(i,j) = (d-\mu-a \mu_1  -b \mu_2, a+b+1)
\]
where $i,j,a,b \in \Z_{\ge 0}$.  This give the triangular region in the plane shown in Figure~\ref{triangularregion}.

\begin{figure}[ht]
\[
{\setlength{\unitlength}{.8pt}\begin{picture}(350,265)
\put(10,20){\line(1,0){330}}
\put(10,20){\line(0,1){240}}
\put(130,40){\circle*{4}}
\put(130,80){\circle*{4}}
\multiput(330,40)(-20,20){8}{\circle*{4}}
\multiput(230,60)(-20,20){6}{\circle*{4}}
\put(130,16){\line(0,1){8}}
\put(330,16){\line(0,1){8}}
\put(230,16){\line(0,1){8}}
\put(310,16){\line(0,1){8}}
\put(127,0){\small $\mu$}
\put(326,0){\small $d-\mu$}
\put(212,0){\small $d{-}\mu{-}\mu_2$}
\put(277,0){\small $d{-}\mu{-}\mu_1$}
\put(328,48){\small$q$}
\put(128,48){\small$p$}
\put(6,40){\line(1,0){8}}
\put(6,60){\line(1,0){8}}
\put(16,37){\small$1$}
\put(16,57){\small$2$}
\Thicklines
\put(130,240){\circle*{4}}
\put(150,220){\circle*{4}}
\put(170,200){\circle*{4}}
\drawline(330,40)(130,80)
\drawline(330,40)(172,198)
\drawline(132,238)(148,222)
\drawline(152,218)(168,202)
\put(310,60){\circle*{5}}
\put(306,68){\small$D_A(q)$}
\put(230,60){\circle*{5}}
\put(226,68){\small$D_B(q)$}
\put(330,40){\circle*{5}}
\dottedline{5}(130,80)(30,100)
\dottedline{5}(130,240)(110,260)
\end{picture}}
\]
\caption{The Triangular Region}
\label{triangularregion}
\end{figure}
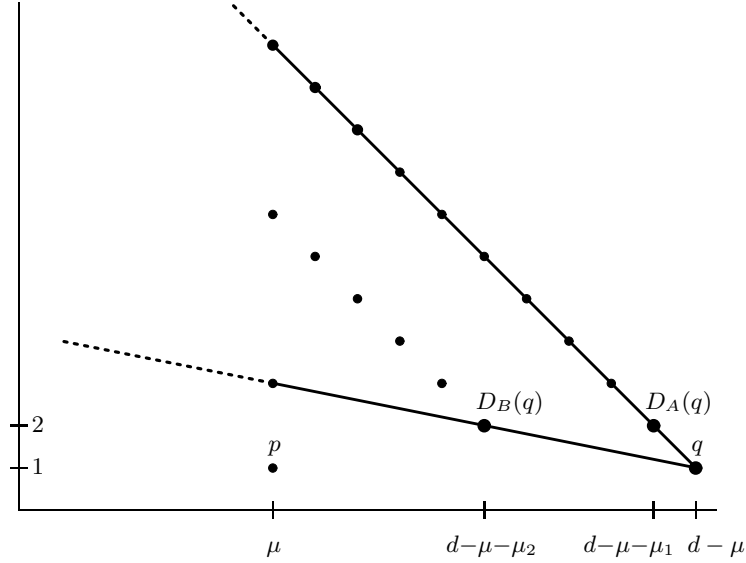
By \cite[Corollary 3.13]{Madsen},
we know that for $i \ge \mu$, the minimal generators
of bidegree $(i,j)$ for $\mathcal{K}$ lie in the triangular region in Figure~\ref{triangularregion}, and  correspond to elements which are mapped to $X^i Y^j(\alpha_{d-\mu_1} Y- \beta_{d-\mu_2} X)$ via $\Omega$. From Corollary \ref{pi}, we deduce that
$\pm D_A^iD_B^j(q)$ gets mapped to this polynomial. Lemma \ref{ax} applied to $q$  concludes with the proof of the claim.
\end{proof}

\begin{remark}\label{rremr}
The hypothesis $\mu_1>0$ is necessary, as otherwise if $d-\mu>\mu+\mu_2-1=2\mu-1,$ we would be able to produce  the infinite family ${D_A}^j(q),\, j=0,1,\dotsb,$ which gets mapped to $(-1)^jY^j(\alpha_{d-\mu_1}Y-\beta_{d-\mu_2}X),$ which clearly cannot be part of a (finite) system of minimal generators of $\cK$. Moreover, thanks to \cite[Theorem 4.6]{CD13}, we know that there are no minimal generators of $\T$-degree $d-\mu$ except for $q$.
\end{remark}

\subsection{The cases $\mu_1=0$ and $0<\mu_1=\mu_2$}
Figures \ref{polygonfig} and \ref{triangularregion} are made under the assumption that $0<\mu_1<\mu_2.$ But what happens if $\mu_1=0$ or $0<\mu_1=\mu_2$? In the first case, the segment defined by $D_A$ becomes parallel to the vertical axis, and an infinite family $D_A^j(q)$ may be produced for all $j\geq0.$ But Theorem \ref{DD} does not hold as explained by  Remark \ref{rremr}. This is because Proposition \ref{indep} does not hold in this case (cf.\ Remark \ref{rrem}). However, one can prove that the family $\{D_B^j(q)\}$ for all those $j$ such that it is defined, is part of a minimal system of generators by modifying Proposition \ref{indep} as follows:
\begin{proposition}
If $\mu_1\geq0$ and the family $\{F_{i_1,j_1},\dotsb, F_{i_\ell,j_\ell}\}\subset\cK$ is such that $\Omega\big(F_{i_k,j_k}\big)=Y^{b_k}(\alpha_{d-\mu_1}Y-\beta_{d-\mu_2}X)$ for $k=1,\dotsb, \ell$, where $b_{k'}\neq b_k$ if $k\neq k'$, then this family is contained in a system of minimal generators of $\cK$.
\end{proposition}
\begin{proof}
Follow the proof of Proposition \ref{indep} until \eqref{qui}. Set $X\mapsto 0$ in that identity, to conclude that there must be $s_0\in\{1,\dotsb, m\}$ such that  $R_{s_0}=\lambda_{s_0}\in\K^\times.$ From here, the proof can be completed as in the proof of Proposition \ref{indep}.
\end{proof}

The case $\mu_1=\mu_2$ corresponds to when the two segments defined by both $D_A$ and $D_B$ in Figure \ref{triangularregion} coincide, and hence the triangular region becomes a segment. In this case,  for all the admissible $j\geq0,$ there are  $j+1$ elements of bidegree $(d-\mu-j\mu_1,j+1)$ in $\cK$ which  get mapped via $\Omega$ to
$X^aY^b(\alpha_{d-\mu_1}Y-\beta_{d-\mu_2}X),\,a+b=j$, and hence thanks toTheorem \ref{DD} they are part of a minimal system of generators of $\cK$.

\section{Comparison with Madsen's Results}\label{6}
Our situation is studied by Madsen in \cite{Madsen}.  For $K = \langle f_{0,d},f_{1,d},f_{2,d}\rangle \subseteq \K[\T]$, the Hilbert-Burch matrix has only two columns, the first of which is $p = (p_{0,\mu},p_{1,\mu},p_{2,\mu})$.  Thus the sequence of Rees algebras \eqref{madsenapproach} simplifies to
\[
\K[\T,\uZ] \twoheadrightarrow \cR(E_1) \twoheadrightarrow \cR(K),
\]
where $\cR(E_1) \simeq \K[\T](-d)^3/(p_{0,\mu},p_{1,\mu},p_{2,\mu}) \K[\T](-d)^3$ by \cite[Sec.~4]{Madsen}.   Unlike \cite{Madsen}, we do not shift by $d$, which explains why we use $\K[\T](-d)^3$.  

In the notation of Section~\ref{4}, $p$ has a $\mu$-basis $A,B$ of degrees $\mu_1+\mu_2 = \mu$.  Thinking of $A,B$ as row vectors, we get a map
\[
E_1 \simeq \K[\T](-d)^3/(p_{0,\mu},p_{1,\mu},p_{2,\mu}) \K[\T](-d)^3 \stackrel{\scriptstyle\big(\begin{smallmatrix}A \\[1pt] B\end{smallmatrix}\big)}{\longrightarrow} R(\mu_1-d)\oplus R(\mu_2-d) = F,
\]
and Madsen notes that $F = E_1^{**}$.  Similar to \eqref{Madsen.sec3}, the inclusion $E_1 \hookrightarrow F$ gives a commutative diagram 
\begin{equation}
\label{Madsen.sec6}
\begin{array}{c}
\SelectTips{cm}{}
\xymatrix@C=15pt@R=12pt{
\K[\T,\uZ] \ar@{->>}[r]^{\Omega} \ar@{->>}[rrd]_\psi 
& \cR(E_1) \ar@{^{(}->}[r] \ar@{->>}[rd]& \K[\T,X,Y] \ar[dr]^\phi\\
&& \cR(K)  \ar@{^{(}->}[r] & \K[\T,s]}
\end{array}
\end{equation}
where the maps $\Omega, \psi,\phi$ are from the diagram \eqref{Reesdiagram} and also feature in Lemma~\ref{maplemma}.  The difference is that we now regard $\Omega$ and $\psi$ as surjections onto their images, which are the Rees algebras $\cR(E_1)$ and $\cR(K)$ respectively.  

Recall that the goal is to understand $\cK = \ker (\psi)$.  Similar to \eqref{sec3.intersection1}, Lemma~\ref{mapcor} implies that
\[
\cK \cR(E_1) = (g\cR(F)) \cap \cR(E_1)
\]
for $g = \alpha_{d-\mu_1}Y - \beta_{d-\mu_2}X$.   Madsen refines this with
\begin{equation}
\label{sec6.intersection2}
\cK_{\ge\mu,*} \cR(E_1) = (g\cR(F))_{\ge\mu,*},
\end{equation}
which follows from \cite[(3.11)]{Madsen}.  This enables Madsen to show that the minimal generators of bidegree $(i,j)$ for $\mathcal{K}$ lie in the triangular region in Figure~\ref{triangularregion}.

These bidegrees  $(i,j)$ correspond to elements which are mapped to $X^i Y^j g$ via $\Omega$. From Corollary \ref{pi}, we deduce that
$\pm D_A^iD_B^j(q)$ gets mapped to this polynomial, so Theorem \ref{DD} can be regarded as an explicit description  of these particular generators.  We do not succeed in covering all the elements predicted by \cite[Theorem 3.9]{Madsen}: there may be some points at the top of the upper edge in Figure~\ref{triangularregion} corresponding to bidegrees where we cannot predict in advance that $D_A^a D_B^b (q)$ is defined. For instance,  when $d = 22$, $\mu = 6$, $\mu_1 = 1$, and $\mu_2 = 5$, there are three open dots at the top of the upper edge  where  our method does not guarantee to produce any element of those bidegrees (see  Figure~\ref{ConfusedFig} in Section~\ref{7}).
  
In Section~\ref{3}, the analog of \eqref{sec6.intersection2} was \eqref{sec3.intersection2}, which we proved by elementary methods in \eqref{2.9.claim}.  However, our methods are not strong enough to prove \eqref{sec6.intersection2}, which is why we rely upon Madsen's results in the proof of Theorem~\ref{DD}.

In Section 3.3 of \cite{Madsen} there is also an algorithm to compute the generators  of $\cK$.  Let us describe this method and explain its relation to the operators $D_A$ and $D_B$ defined in Section~\ref{5}.  Using a $\mu$-basis $\{b_1,b_2\}$ of $B$, define four  polynomials: 
\begin{align*}
p^A_i &= b_i \cdot A \in \K[\T], \quad i = 1,2\\
\rho^A_i &= b_i \cdot (Z_0,Z_1,Z_2), \quad i = 1,2.
\end{align*}
Since $b_i$ is a syzygy on $B$ and $\Omega(Z_i) = A_{i,\mu_1} X + B_{i,\mu_2}Y$, an easy calculation yields
\begin{equation}
\label{omegarho}
\Omega(\rho^A_i) = X \Omega(p^A_i).
\end{equation}
More generally, if $F_{i,j} \in \K[\T,\uZ]$ can be written
\begin{equation}
\label{Fijp}
F_{i,j} = h_1 p^A_1 +  h_2 p^A_2,
\end{equation}
then (using different notation) Madsen defines
\[
F_{i,j}^A = h_1 \rho^A_1 +  h_2 \rho^A_2.
\]
Note that $\deg(F_{i,j}^A) = (i-\mu_1,j+1)$ since $A$ has degree $\mu_1$, and \eqref{omegarho} implies that $\Omega(F_{i,j}^A) = X \Omega(F_{i,j})$.
Since $\Omega(D_A(F_{i,j})) = X \Omega(F_{i,j})$ by Proposition~\ref{DaDb},  $F_{i,j}^A$ and $D_A(F_{i,j})$ differ by a multiple of $p$ (see Proposition~\ref{kerOmega}).  

In Proposition 3.15 of \cite{Madsen}, Madsen proves that $\langle p_1^A,p_2^A\rangle_{\mu+\mu_1-1} = \K[\T]_{\mu+\mu_1-1}$.  This tells us that when $i \ge \mu+\mu_1-1$, \eqref{Fijp} holds and hence $F_{i,j}^A$ is defined.  In contrast, the definition of $D_A$ requires that $i \ge \mu+\mu_2-1$.  

Madsen also has an analog of our $D_B$ operator that is defined using a $\mu$-basis of $A$.  Here, $F_{i,j}^B$ has degree $(i-\mu_2,j+1)$ and is defined for $i \ge \mu+\mu_2-1$, the same as for $D_B$.  Furthermore $\Omega(F_{i,j}^B) = Y \Omega(F_{i,j})$, so that $F_{i,j}^B$ and $-D_B(\deg(F_{i,j})$ differ by a multiple of $p$ (remember the minus sign in Proposition~\ref{DaDb}).

By starting with $q \in \cK_{d-\mu,1}$ and applying the $\{\}^A$ and $\{\}^B$ operators, Madsen constructs all minimal generators of $\cK$ with $i \ge \mu$ \cite[Corollary 3.17]{Madsen}.  There is also an interpretation in terms of Sylvester forms \cite[Proposition 3.18]{Madsen}.  

Because of the restriction that $i \ge \mu+\mu_2-1$, our operators $D_A$ and $D_B$ do not give all all minimal generators of $\cK$ with $i \ge \mu$.  As noted above, Figure~\ref{ConfusedFig} shows what can happen.  The generators we miss in this figure all require $D_A$, which requires $i \ge \mu+\mu_2-1$.  Madsen's $\{\}^A$ only requires $i \ge \mu +\mu_1-1$, which explains the success of the methods in \cite{Madsen}.

An intriguing observation is that we start with $\mathbf{f} = (f_{0,d},f_{1,d},f_{2,d})$ with $\mu$-basis $\{p,q\}$ and then use a $\mu$-basis $\{A,B\}$ of $p$ to construct elements of $\cK$. In \cite{Madsen}, Madsen uses $\mu$-bases of $A$ and $B$ to construct further elements of $\cK$.  Is it possible that repeatedly taking $\mu$-bases could lead to a complete description of the minimal generators of $\cK$?

\section{Lifts of Minimal Generators}\label{7}

Recall from Section \ref{2} that for $\cI$, the minimal generators consist of the minimal generators of $\cI'$ together with the generators
\begin{equation}
\label{Psigens}
\Psi^t_{i_\ell,j_\ell,k_\ell+1}, \ 0 \le t \le s_\ell
\end{equation}
described in Theorem~\ref{mtm}. We also know that $s_\ell = 0$ when $i_\ell > 0$, so that when $i_\ell > 0$, \eqref{Psigens} becomes
\[
\Psi^0_{i_\ell,j_\ell,k_\ell+1} \in \cI_{d-\mu-\mu_1 j_\ell - \mu_2 k_\ell,j_\ell+k_\ell+1}
\]
since $s_\ell = 0$ implies that $i_\ell = d-\mu-\mu_1 j_\ell - \mu_2 k_\ell$.  This bidegree lies in triangle obtained by extending the dotted lines in Figure~\ref{triangularregion} to the $y$-axis.

The map $\Gamma : \K[\T,\uZ] \to \K[\T,\X,\Y]$ defined in \eqref{Gama} satisfies $\Gamma(\cK) \subseteq \cI$.  For $f \in \cK$, we call $\Gamma(f)$ the \emph{lift} of $f$.

\subsection{Lifting \boldmath{$q$}}

We now show how to lift $q$.

\begin{lemma}\label{liftq}
For $q = q_{0,d-\mu} Z_0 + q_{1,d-\mu} Z_1 + q_{2,d-\mu} Z_2 \in \cK_{d-\mu,1}$, we have $$\Gamma(q) = \Psi^0_{d-\mu,0,1} \in \cI_{d-\mu,1}.$$
\end{lemma}

\begin{proof}
We again drop subscripts indicating degree.  By definition,
\begin{align*}
\Gamma(q) &=q_0 (\ua_0\cdot\X+\ub_0\cdot\Y) +, q_1(\ua_1\cdot\X+\ub_1\cdot\Y) + q_2(\ua_2\cdot\X+\ub_2\cdot\Y)\\
&= (q_0 \ua_0 + q_1 \ua_1 +q_2 \ua_2 )\cdot \X +
(q_0 \ub_0 + q_1 \ub_1 +q_2 \ub_2 )\cdot \Y.
\end{align*}
One of the minimal generators of $S_{\mu_1,\mu_2,d}$ from \eqref{Smu}  is $\mathbf{v}_\ell = (d-\mu,0,0)$, where $s_\ell = 0$.  If we pick
\[
A^0_{d-\mu,0,1} = (q_0 \ub_0 + q_1 \ub_1 +q_2 \ub_2 )\cdot \Y,
\]
then
\[
A^0_{d-\mu,0,1}(\T,\T^{\mu_1},\T^{\mu_2}) = (q_0 \ub_0 + q_1 \ub_1 +q_2 \ub_2 )\cdot \T^{\mu_2} = (q_0 B_0 + q_1B_1+q_2B_2) = -\alpha_{d-\mu_1},
\]
where the last equality is by \eqref{xio}.  Then one computes that
\[
B^0_{d-\mu,1,0} = -(q_0 \ua_0 + q_1 \ua_1 +q_2 \ua_2 )\cdot \X
\]
satisfies $B^0_{d-\mu,1,0}(\T,\T^{\mu_1},\T^{\mu_2}) = -\beta_{d-\mu_2}$ again by \eqref{xio}, and
\begin{align*}
\Psi^0_{d-\mu,0,1} &=
A^0_{d-\mu,0,1} - B^0_{d-\mu,1,0}\\ &= (q_0 \ua_0 + q_1 \ua_1 +q_2 \ua_2 )\cdot \X +
(q_0 \ub_0 + q_1 \ub_1 +q_2 \ub_2 )\cdot \Y,
\end{align*}
which is the above formula for $\Gamma(q)$.
\end{proof}

\subsection{Lifting Other Generators}
The general strategy for lifting minimal generators from $\cK$ to $\cI$ is to work mod $\cI'$, whose minimal generators are described in Proposition~\ref{2.2}.  Since $\cI' = \ker(\Phi')$, studying $H \in \cI$ mod $\cI'$ means working with $\Phi'(F) \in \K[\T,X,Y]$. For a minimal generator $\Psi^t_{i_\ell,j_\ell,k_\ell+1} \in \cI$, the following result tells us exactly what its image in $\K[\T,X,Y]$ looks like.

\begin{proposition}
\label{phiipPsi}
The generator  $\Psi^t_{i_\ell,j_\ell,k_\ell+1}$ gets mapped via $\Phi'$ to  the element $\big(\alpha_{d-\mu_1} Y - \beta_{d-\mu_2} X\big) X^{j_\ell} Y^{k_\ell} T_0^t T_1^{s_\ell-t}$.
\end{proposition}

\begin{proof}
Recall from \eqref{cozi} that
\[
\Psi^t_{{i_\ell,j_\ell,k_\ell+1}}=A^t_{{i_\ell,j_\ell,k_\ell+1}}-B^t_{i_\ell,j_\ell+1,k_\ell},
\]
where $A^t_{{i_\ell,j_\ell,k_\ell+1}},\,B^t_{i_\ell,j_\ell+1,k_\ell}$ are trihomogeneous as indicated by their subscripts.  Also, by \eqref{ambi},
\[
A^t_{i_\ell,j_\ell,k_\ell+1}(\T,\T^{\mu_1},\T^{\mu_2}) =  \alpha_{d-\mu_1} T_0^t T_1^{s_\ell-t},
\]
and by the proof of Lemma \ref{util2p},
\[
B^t_{i_\ell,j_\ell+1,k_\ell}(\T,\T^{\mu_1},\T^{\mu_2}) =  \beta_{d-\mu_2} T_0^t T_1^{s_\ell-t}.
\]
These formulas together with \eqref{phiip} and the trihomogeneity of $A^t,B^t$ imply that
\begin{align*}
\Phi'(\Psi^t_{i_\ell,j_\ell,k_\ell+1}) &= A^t_{i_\ell,j_\ell,k_\ell+1}(\T,\T^{\mu_1}X,\T^{\mu_2}Y) - B^t_{i_\ell,j_\ell+1,k_\ell}(\T,\T^{\mu_1}X,\T^{\mu_2}Y) \\
&=  X^{j_\ell} Y^{k_\ell + 1} \alpha_{d-\mu_1} T_0^t T_1^{s_\ell-t}
-  X^{j_\ell+1} Y^{k_\ell}\beta_{d-\mu_2} T_0^t T_1^{s_\ell-t} \\
&= \big(\alpha_{d-\mu_1} Y - \beta_{d-\mu_2} X\big) X^{j_\ell} Y^{k_\ell} T_0^t T_1^{s_\ell-t}.\qedhere
\end{align*}
\end{proof}

For $F \in \cK$, applying the above strategy to its lift
$\Gamma(F)$ mod $\cI'$ means working with $\Phi'(\Gamma(F)) = \Omega(F)$ by Lemma~\ref{maplemma}.  For the minimal generators of $\cK$ identified in Theorem~\ref{DD}, this leads to the following result.

\begin{theorem}\label{DDlift}
Suppose $a,b$ satisfy $a \ge 0$ and either $b \ge 1, d-\mu-a\mu_1-b\mu_2\geq\mu-1$ or  $b = 0, d-\mu-a\mu_1\geq\mu+\mu_2-\mu_1-1$.  Then we have a minimal generator
\[
D_A^aD_B^b(q) \in \cK_{d-\mu-a\mu_1-b\mu_2,a+b+1}
\]
whose lift to $\cI$ satisfies
\[
\Gamma(D_A^aD_B^b(q)) \equiv (-1)^{b+1} \Psi^0_{d-\mu-a\mu_1-b\mu_2,a,b+1} \bmod \cI'.
\]
\end{theorem}

\begin{proof}
By \eqref{ki}, we have $\Omega(D_A^aD_B^b(q))=(-1)^{b+1}X^bY^a(\alpha_{d-\mu_1} Y -\beta_{d-\mu_2} X)$.  As noted above, this implies
\[
\Phi'\big(\Gamma(D_A^aD_B^b(q))\big)=(-1)^{b+1} X^aY^b\big(\alpha_{d-\mu_1} Y -\beta_{d-\mu_2} X \big).
\]
Since $\Phi'( \Psi^0_{d-\mu-a\mu_1-b\mu_2,a,b+1}) = X^aY^b(\alpha_{d-\mu_1} Y -\beta_{d-\mu_2} X)$ by Proposition~\ref{phiipPsi}, the theorem follows immediately.
\end{proof}

Here is an example that gives a picture of which minimal generators of $\cK$ are involved in Theorem~\ref{DDlift}.

\begin{example}
\label{fromConfused}
When $d = 22$, $\mu = 6$, $\mu_1 = 1$, and $\mu_2 = 5$, the part of Figure~\ref{triangularregion} with $i \ge \mu$ is shown in Figure~\ref{ConfusedFig}.    The large dots in the figure show $q$, $D_A(q)$, and $D_B(q)$.

{\setlength{\unitlength}{.8pt}

\begin{figure}[h]
\[
\begin{picture}(350,265)
\put(10,20){\line(1,0){330}}
\put(10,20){\line(0,1){240}}
\put(130,40){\circle*{4}}
\put(130,80){\circle*{4}}
\multiput(330,40)(-20,20){8}{\circle*{4}}
\multiput(230,60)(-20,20){6}{\circle*{4}}
\put(130,16){\line(0,1){8}}
\put(330,16){\line(0,1){8}}
\put(115,0){$\mu = 6$}
\put(293,0){$d-\mu =16$}
\put(328,48){$q$}
\put(128,48){$p$}
\Thicklines
\put(130,240){\circle{4}}
\put(150,220){\circle{4}}
\put(170,200){\circle{4}}
\drawline(330,40)(130,80)
\drawline(130,80)(130,237)
\drawline(330,40)(172,198)
\drawline(132,238)(148,222)
\drawline(152,218)(168,202)
\put(310,60){\circle*{5}}
\put(306,68){$D_A(q)$}
\put(230,60){\circle*{5}}
\put(226,68){$D_B(q)$}
\put(330,40){\circle*{5}}
\end{picture}
\]
\caption{$d = 22$, $\mu = 6$, $\mu_1 = 1$, $\mu_2 = 5$}
\label{ConfusedFig}
\end{figure}
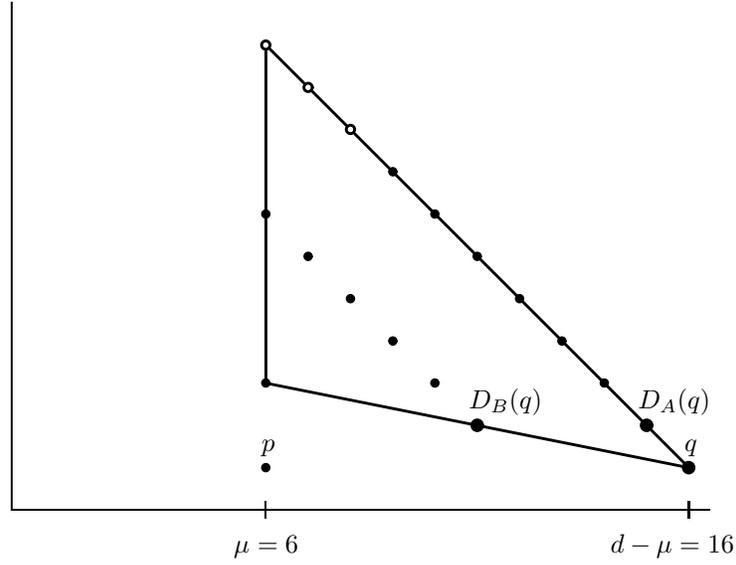}

By Madsen's results \cite{Madsen}, the dots (solid and open) correspond to bidgrees $(i,j)$ of all minimal generators of $\cK$ with $i \ge \mu$.  The inequalities of Theorem~\ref{DDlift} become
\begin{itemize}
\item[] $b \ge 1$, $a \ge 0$, and $a+5b \le 11$.
\item[] $b = 0$ and $0 \le a \le 7$.
\end{itemize}
In fact, $b = 0$ gives the eight solid dots on the upper edge of the triangular region, and $b \ge 1$ gives the remaining solid dots in the region.  The three open dots at the top of the upper edge correspond to bidegrees where our methods  cannot guarantee  that $D_A^a D_B^b (q)$ is defined, in contrast with Madsen's results.
\end{example}


\end{document}